\documentclass[a4paper,10pt,reqno]{amsart} 
\usepackage[foot]{amsaddr}
\usepackage{amsfonts, amssymb, amsmath, wrapfig, amsthm, palatino,
  euler, microtype, enumitem}

\usepackage[colorlinks=true,pagebackref=true]{hyperref}
\usepackage{epsfig}
\usepackage{textcomp}
\usepackage{url}

\theoremstyle{plain}
\newtheorem{theorem}{Theorem}
\newtheorem{lemma}{Lemma}

\newcommand\dynkin[8]{\vcenter{\vbox{\vfill\hbox{$#1#3#4#5#6#7#8$}
\nointerlineskip\vskip 3pt
\hbox{$\phantom{#1}\phantom{#3}#2$\hfil}\vfill}}}

\newcommand\rmA{\operatorname{A}}
\newcommand\ad{\operatorname{ad}}
\newcommand\CC{\operatorname{C}}
\newcommand\Char{\operatorname{char}}
\newcommand\DD{\operatorname{D}}
\newcommand\diag{\operatorname{diag}}
\newcommand\EE{\operatorname{E}}

\newcommand\FF{\operatorname{F}}
\newcommand\Fix{\operatorname{Fix}}
\newcommand\GL{\operatorname{GL}}
\newcommand\GSp{\operatorname{GSp}}
\newcommand\Lie{\operatorname{Lie}}
\newcommand\ppf{\operatorname{pf}}
\newcommand\sic{\operatorname{sc}}
\newcommand\SL{\operatorname{SL}}
\newcommand\SO{\operatorname{SO}}
\newcommand\Sp{\operatorname{Sp}}
\newcommand\tent{\operatorname{tent}}
\newcommand\tr{\operatorname{tr}}
\newcommand\Tran{\operatorname{Tran}}

\newcommand\fo[1]{{\varpi}_{#1}}
\newcommand\Int{{\mathbb Z}}

\newcommand\eps{\varepsilon}

\def\leq{\leqslant}
\def\le{\leqslant}
\def\geq{\geqslant}
\def\ge{\geqslant}

\begin{document}

\dedicatory{To St. Petersburg remarkable algebraist\\
Sergei Vladimirovich Vostokov,\\
a teacher, a colleague, and a friend}

\title[Normalizer of the Chevalley group of type ${\mathrm E}_7$]
{Normalizer\\ of the Chevalley group of type ${\mathrm E}_7$}

\author
{Alexander~Luzgarev}

\address{Saint-Petersburg State University\\
7/9 Universitetskaya nab.\\
St. Petersburg, 199034 Russia.}

\email{a.luzgarev@spbu.ru}

\author
{Nikolai~Vavilov}

\address{Saint-Petersburg State University\\
7/9 Universitetskaya nab.\\
St. Petersburg, 199034 Russia.}

\email{nikolai-vavilov@yandex.ru}

\keywords
{Chevalley groups, elementary subgroups, minimal modules, 
invariant forms, decomposition of unipotents, root elements, 
highest weight orbit}

\begin{abstract}
We consider the simply connected Chevalley group 
$G(\mathrm E_7,R)$ of type $\mathrm E_7$ in the 56-dimensional 
representation. The main objective of the paper is to prove
that the following four groups coincide: the normalizer of
the elementary Chevalley group $E(\mathrm E_7,R)$, the normalizer 
of the Chevalley group $G(\mathrm E_7,R)$ itself, the transporter
of $E(\mathrm E_7,R)$ into $G(\mathrm E_7,R)$, and the extended
Chevalley group $\overline G(\mathrm E_7,R)$. This holds over
an arbitrary commutative ring $R$, with all normalizers and
transporters being calculated in $\mathrm{GL}(56,R)$. Moreover,
we characterize $\overline G(\mathrm E_7,R)$ as the stabilizer of
a system of quadrics. This last result is classically known over
algebraically closed fields, here we prove that the corresponding
group scheme is smooth over $\mathbb Z$, which implies that it 
holds over arbitrary commutative rings. These results are one 
of the key steps in our subsequent paper, dedicated to the
overgroups of exceptional groups in minimal representations. 
\end{abstract}

\thanks
{The main results of the present paper were proven in the
framework of the RSF project 14-11-00297}


\maketitle

The most natural way to study general orthogonal group, is to
represent it as the stabilizer of a quadric. In the present paper,
we establish a similar geometric characterization of the 
normalizer of the simply connected Chevalley group
$G_{\sic}(\EE_7,R)$ as the stabilizer of the intersection of 133
quadrics in a 56-dimensional space, and prove that the above 
normalizer coincides with the normalizer of the elementary 
Chevalley group $E_{\sic}(\EE_7,R)$. 

The present work is a direct sequel of our papers~\cite{Vavilov_Luzgarev_norme6, 
Luzgarev_F4}, where a similar exercise was carried through 
for the groups of types $\EE_6$ and~$\FF_4$.

\section{Introduction}

In the paper~\cite{Luzgarev_overgroups} 
(see also \cite{Luzgarev_F4, Luzgarev_diss}) the second author has started 
to carry over the results by the first author and Victor Petrov 
\cite{Vavilov_Petrov_Ep, Vavilov_Petrov_EO, Petrov_unitary} on overgroups
of classical groups in vector representations, to the exceptional groups 
$E(\EE_6,R)$ and~$E(\EE_7,R)$, in minimal representations. From the 
very start, it became apparent, that one the key steps necessary to 
carry through a reduction proof in the spirit of the cited papers, 
would be an {\it explicit\/} calculation of the normalizer of the
above groups in the corresponding general linear group, $\GL(27,R)$ 
or $\GL(56,R)$, respectively. 

In our previous paper \cite{Vavilov_Luzgarev_norme6} we have 
completely solved this problem for the group $E(\EE_6,R)$, whereas in 
\cite{Luzgarev_F4} this problem is solved for the group $E(\FF_4,R)$. 
In the present paper, we consider in the same spirit the group of
type $\EE_7$.

More precisely, in \S\ref{sec:construction} we explicitly construct an
ideal $I$ in the ring of integer polynomials $\Int[x_1,\ldots,x_{56}]$,
generated by 133 quadratic forms $f_1,\ldots,f_{133}$, which has the
following property. Denote by $\Fix_R(I)$ the set of $R$-linear 
transformations, preserving the ideal $I$, see~\S\ref{sec:construction} 
for the precise definitions. 

The first main objective of the present paper, is to prove the 
following result. Here, $G_I$ denotes the affine group scheme
such that $G_I(R)=\Fix_R(I)$.

\begin{theorem}\label{thm:g_i}
There is an isomorphism $G_I\cong\overline G_{\sic}(\EE_7,-)$
of affine groups schemes over $\Int$.
\end{theorem}

This result can be viewed as an explicit description of the
{\it extended\/} simply connected Chevalley--Demazure group 
scheme $\overline G_{\sic}(\Phi,-)$ of type $\Phi=\EE_7$, 
by equations. This scheme was constructed in~\cite{Berman_Moody}, 
see also \cite{Vavilov_weight, Vavilov_tor90, Vavilov_weightnew}
and~\S\ref{sec:extended_group} below.
For $\Phi=\EE_7$ the most straightforward way to visualize the
scheme $\overline G_{\sic}(\Phi,-)$ is to view it as the 
Levi factor of the parabolic subscheme of type $P_8$ in 
$G_{\sic}(\EE_8,-)$, where
$G_{\sic}(\Phi,-)$~--- is the usual simply connected 
Chevalley--Demazure groups scheme of type $\Phi$. We refer 
the reader to \cite{Matsumoto} as for the scheme-theoretic 
definition of parabolic subgroups and their Levi factors,
see also~\cite{Vavilov_weightnew} for the above identification
itself.

Our results are intimately related to the description of 
$G_{\sic}(\EE_7,R)$ as the stabilizer of a system of four-linear
forms on $V=V(\varpi_7)$. Namely, in~\cite{Luzgarev_e7_invariants}
we gave a new construction of a four-linear form 
$f\colon V\times V\times V\times V\to R$ and a symplectic form
$h\colon V\times V\to R$, invariant under the action
of the group $G_{\sic}(\EE_7,R)$. We reproduce the construction 
of the form $f$ in~\S\ref{sec:invariants}. The bulk of our
system of quadratic forms consists of the second partial
derivatives of the [regular part of] the form~$f$.

It turns out that $G_{\sic}(\EE_7,R)$ is precisely the group
of linear transformations preserving both $f$ and $h$:
\begin{align*}
G_{(f,h)}(R) = \{&g\in\GL(56,R)\mid f(gu,gv,gw,gz)=f(u,v,w,z),
\\
&h(gu,gv)=h(u,v)\text{ for all } u,v,w,z\in V\}.
\end{align*}
It is only marginally more complicated to describe the 
extended group $\overline{G}_{\sic}(\EE_7,R)$ in terms of
the forms $f$ and $h$. Namely, let
\begin{align*}
\overline{G}_{(f,h)}(R) &= \{g\in\GL(56,R)\mid\text{there exist
  $\varepsilon,\varepsilon'\in R^*$, $c_2,c_3,c_4\in R$}
\\
&\qquad \text{such that }f(gu,gv,gw,gz)=\varepsilon f(u,v,w,z)
\\
&\qquad\qquad + c_2 h(u,v)h(w,z) + c_3 h(u,w)h(v,z) + c_4 h(u,z)h(v,w)
\\
&\qquad \text{and }h(gu,gv)=\varepsilon'h(u,v)\text{ for all }
u,v,w,z\in V\}.
\end{align*}

\begin{theorem}\label{thm:g_fh}
There are isomorphisms
$G_{(f,h)}\cong G_{\sic}(\EE_7,-)$,
$\overline G_{(f,h)}\cong\overline G_{\sic}(\EE_7,-)$
of affine groups schemes over $\Int$.
\end{theorem}

This theorem readily implies that the above definition of
the extended group $\overline{G}_{\sic}(\EE_7,R)$ can be
simplified as follows. Namely, Lemma~\ref{lemma:Petrov} 
asserts that
\begin{align*}
\overline{G}_{(f,h)}(R) &= \{g\in\GL(56,R)\mid\text{there exists
an $\eps\in R^*$ such that}
\\
&\qquad f(gu,gv,gw,gz)=\eps f(u,v,w,z)
\\
&\qquad \text{and }h(gu,gv)=\eps^2h(u,v)\text{ for all }
u,v,w,z\in V\}.
\end{align*}

Now, let $E,F$ be two subgroups of a group $G$. Recall that
the {\bf transporter} of the subgroup $E$ to the subgroup 
$F$ is the set
$$ 
\Tran_G(E,F)=\{g\in G\mid E^g\le F\}. 
$$
Actually, we mostly use this notation in the case where $E\le F$,
and then
$$ 
\Tran_G(E,F)=\{g\in G\mid [g,E]\le F\}. 
$$
In the sequel, we only work with the simply connected groups
and omit the subscript in the notation $G_{\sic}(\Phi,R)$. 
By $E(\Phi,R)\le G(\Phi,R)$ we denote the elementary Chevalley
group. Now we are all set to state the main result of the
present paper. Observe, that all normalizers and transporters
here are taken in the general linear group $\GL(56,R)$.

\begin{theorem}\label{thm:normaliser}
Let $R$ be an arbitrary commutative ring. Then
$$ 
N(E(\EE_7,R))=N(G(\EE_7,R))=\Tran(E(\EE_7,R),G(\EE_7,R))=G_I(R).
$$
\end{theorem}

The interrelation of Theorems~\ref{thm:g_i} and~\ref{thm:normaliser}
and the general scheme of their proof are exactly the same, as in 
our previous paper\cite{Vavilov_Luzgarev_norme6}, and some
familiarity with \cite{Vavilov_Luzgarev_norme6} (at least with the 
introduction and~\S5) would be {\it extremely\/} useful to
facilitate understanding the proofs in the present paper. 

Observe, that after the publication of
\cite{Vavilov_Luzgarev_norme6} its subject matter became
{\it unexpectedly\/} pertinent. Namely, recently Elena Bunina 
reconsidered one of the central classical problems of the whole
theory, description of [abstract] automorphisms of Chevalley
groups, without any such simplifying assumptions as $R$ being
Noetherian, or $2$ being invertible in $R$. For local rings
she {\it almost\/} succeeded in proving that all automorphisms
of the group $E(\Phi,R)$ are standard, see~\cite{Bunina1}, etc. 
Namely, she established that an arbitrary automorphism of the 
adjoint elementary Chevalley group is the product of 
ring, inner and graph automorphisms. There is a catch, though,
that with her approach the inner automorphisms are taken not
in the adjoint Chevalley group $G_{\ad}(\Phi,R)$ itself, but rather
in the corresponding general linear group ${\mathrm{GL}}(n,R)$. 
In this context, the fact that the abstract and algebraic normalizers
coincide, means precisely that all such conjugations are genuine
inner automorphisms. 

This means that modulo the results of \cite{Bunina1} an analogue
of the results of \cite{Vavilov_Luzgarev_norme6} and the present
paper, for adjoint {\it representations\/} would then imply 
that all automorphisms of Chevalley groups of types $\EE_l$ 
over local rings --- and thus also arbitrary commutative rings
--- are standard in the usual sense. We are convinced that 
our results on the equations in adjoint representations
\cite{Luzgarev2014, Luzgarev_Petrov_Vavilov} allow to
obtain the requisite results for the adjoint case. In cooperation
with Elena Bunina, we hope to work out the details shortly.

The paper is organized as follows. In \S2 we recall the basic
notation pertaining to the extended Chevalley group of type
$\EE_7$. In~\S3 we discuss the invariant four-linear forms, and
in~\S4 we construct an invariant system of quadrics, which
in this case is {\it significantly\/} trickier than in the case 
of $\EE_6$. In~\S5 we prove that this system of quadrics is 
indeed invariant. The technical core of the paper are \S\S6--10, 
which are directly devoted to the proof of Theorems~\ref{thm:g_i},
\ref{thm:g_fh} and~\ref{thm:normaliser}. Due to the limited space, 
we do not explicitly list the resulting equations here, this 
will be done in a subsequent publication.


\section{Extended Chevalley group of type $\EE_7$}\label{sec:extended_group}

We refer the reader to \cite{Vavilov_Luzgarev_norme6} as for the
general context of the present paper, and further references. 
In the papers \cite{Matsumoto, Stein_generators, Vavilov_structure,
Vavilov_third_look, Vavilov_Plotkin} one can find many further
details pertaining to Chevalley groups over rings, and many further
related references.

Nevertheless, to fix the requisite notation, for reader's
convenience below we reproduce with minor modifications~\S1 
of~\cite{Vavilov_Luzgarev_norme6}.

Let $\Phi$ be a reduced irreducible root system of rank $l$ 
(in the main body of the paper we assume that $\Phi=\EE_7$), 
and $P$ be a lattice intermediate between the root lattice $Q(\Phi)$ 
and the weight lattice $P(\Phi)$. We fix and order on $\Phi$
and denote by $\Pi=\{\alpha_1,\ldots,\alpha_l\}$,
$\Phi^+$ and $\Phi^-$ the corresponding sets of fundamental, positive, 
and negative roots. Our numbering of the fundamental roots 
follows~\cite{Bourbaki46}. By $\delta$ we denote the maximal root 
of the system $\Phi$ with respect to this order. For instance, 
for $\Phi=\EE_7$ we have $\delta=\dynkin2234321{}$. Denote by
$P(\Phi)_{++}$ the set of dominant weights with respect to this 
order. Recall that it consists of all nonnegative integral linear 
combinations of the fundamental weights $\fo{1},\ldots,\fo{l}$,
for this order. Finally, $W=W(\Phi)$ denotes the Weyl group
of the root system $\Phi$.

Further, let $R$ be a commutative ring with 1. It is classically known
that, starting with this data, one can construct the
{\it Chevalley group\/} $G_P(\Phi,R)$, which is the group of $R$-points
of an affine group scheme $G_P(\Phi,{-})$, known as the 
{\it Chevalley--Demazure scheme\/}. For the problems we consider, 
it suffices to limit ourselves with the simply connected
(alias, universal) groups, for which $P=P(\Phi)$. For the simply 
connected groups we usually omit any reference to the lattice $P$ 
and simply write $G(\Phi,R)$ or, when we wish to stress that the 
group in question is simply connected, $G_{\sic}(\Phi,R)$. The
adjoint group, for which $P=Q(\Phi)$, is denoted by $G_{\ad}(\Phi,R)$.

Fix a split maximal torus $T(\Phi,R)$ in $G(\Phi,R)$ and a parametrization 
of the unipotent root subgroups $X_{\alpha}$, $\alpha\in\Phi$, elementary
with respect to this torus. Let $x_{\alpha}(\xi)$ be the elementary
unipotent element corresponding to  $\alpha\in\Phi$ and $\xi\in R$ in
this parametrization. The group $X_{\alpha}=\{x_{\alpha}(\xi),\xi\in R\}$ 
is called an (elementary) {\it root subgroup\/}, and the group
$E(\Phi,R)=\langle X_{\alpha},\alpha\in\Phi\rangle$ generated by all 
elementary root subgroups is called the (absolute) {\it elementary
subgroup\/} of the Chevalley group $G(\Phi,R)$.

As a matter of fact, apart from the usual Chevalley group, we also
consider the corresponding {\it extended\/} Chevalley group 
$\overline G(\Phi,R)$, which plays the same role with respect to 
$G(\Phi,R)$ as the general linear group $\GL(n,R)$ plays with 
respect to the special linear group $\SL(n,R)$. {\it Adjoint\/} 
extended groups were constructed in the original paper by Chevalley
\cite{Chevalleyrus}. It is somewhat harder to construct
{\it simply connected\/} extended groups because, unlike the adjoint 
case, here one must increase the dimension of the maximal torus. 
A unified elementary construction was only proposed by Berman and
Moody in \cite{Berman_Moody}. However, for the case of
$\overline G_{\sic}(\EE_7,R)$ that we consider in the present 
paper, this group can be naturally viewed as a subgroup of the
usual Chevalley group $G_{\sic}(\EE_8,R)$, viz.
$$ 
\overline G_{\sic}(\EE_7,R)=G_{\sic}(\EE_7,R)\cdot T_{\sic}(\EE_8,R). 
$$

In the majority of the existing constructions, the Chevalley group 
$G(\Phi,R)$ arises together with an action on the {\it Weyl
module\/} $V=V(\omega)$, for some dominant weight $\omega$. Denote 
by $\Lambda=\Lambda(\omega)$ the {\it multiset\/} of weights of
the module $V=V(\omega)$ {\it with multiplicities\/}. In the present
paper we consider the group $G(\EE_7,R)$ in the minimal representation
with the highest weight $\varpi_7$. This is a microweight 
representation, in particular the multiplicities of all weights are 
equal to 1. Fix an admissible base $v^{\lambda}$, $\lambda\in\Lambda$, 
of the module $V$. We conceive a vector $a\in V$, 
$a=\sum v^{\lambda}a_{\lambda}$, as a coordinate {\it column\/} 
$a=(a_{\lambda})$, $\lambda\in\Lambda$.

\begin{figure}[h]
\begin{center}
\includegraphics{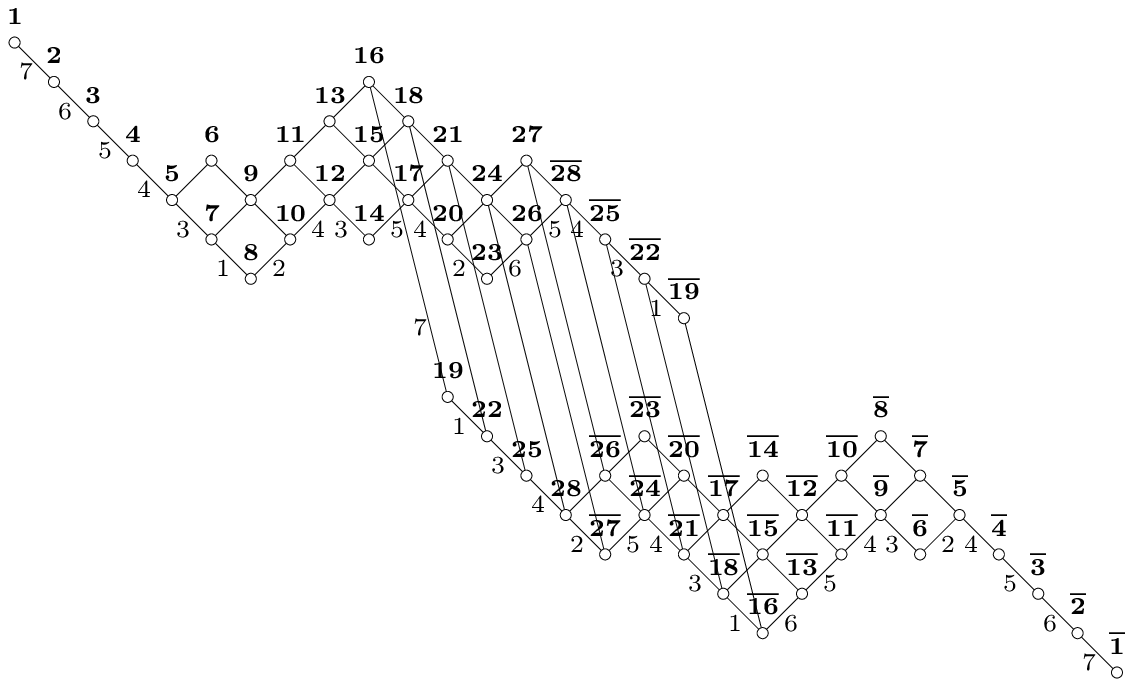}
\caption{Weight diagram $(\EE_7,\varpi_7)$: natural numbering.}
\end{center}
\end{figure}

In Figure~1 we reproduce the weight diagram of the representation
$(\EE_7,\varpi_7)$, together with the {\it natural\/} numbering
of weights, used in the sequel. In this numbering the weights 
are listed according to the order determined by the fundamental
root system $\Pi$. On the picture, the highest weight is the 
left-most one. The weight diagram of the representation 
$(\EE_7,\varpi_7)$ is symmetric, and this symmetry is reflected
in the numbering, the weights are numbered as 
$1,2,\dots,28,-28,\dots,2,1$. Often, to save space we 
write $\overline{n}$ instead of $-n$. We refer the reader to
\cite{Vavilov_Luzgarev_miniE7} for lists of weight in the
Dynkin form, and in the hyperbolic form, as well as other common 
numberings. 

Recall, that in the weight diagram two weights are joined by an
edge if their difference is a {\it fundamental\/} root. The weight
graph is constructed similarly, only that now two weights are
joined by an edge provided their difference is a {\it positive\/} 
root. In the sequel we denote by $d(\lambda,\mu)$ the distance
between two weights $\lambda$ and $\mu$ in the weight graph.
In other words, $d(\lambda,\mu)=0$ if $\lambda=\mu$; $d(\lambda,\mu)=1$
if $\lambda-\mu\in\Phi$; $d(\lambda,\mu)=2$ if $\lambda\neq\mu$,
$\lambda-\mu\not\in\Phi$, $\lambda-\mu$ is the sum of two roots 
of $\Phi$; and finally, $d(\lambda,\mu)=3$ otherwise. 

The above realization of the representation  $(\EE_7,\varpi_7)$ 
as an internal Chevalley module inside the Chevalley group of
type $\EE_8$ provides a natural identification of the set of
weights $\Lambda$ with the set of roots of the root system 
$\EE_8$, in whose expansion with respect to the fundamental roots 
the root $\alpha_8$ occurs with the coefficient $1$. Obviously, 
the roots of the root system $\EE_7$ itself are identified with
those roots of $\EE_8$ in whose expansion $\alpha_8$ occurs with
the coefficient 0. There is a unique root of $\EE_8$, in whose
expansion $\alpha_8$ occurs with the coefficient $2$: it is the
maximal root $\delta = \dynkin23465432$. In the sequel we always 
view both the roots of $\EE_7$ and the wights of our representation
as the roots of $\EE_8$. We denote by $(\cdot,\cdot)$ the natural
inner product defined on the linear span of $\EE_8$. It is
convenient to normalize it in such a way that all roots have 
length $1$. Then for any $\alpha,\beta\in\EE_8$ the inner product
$(\alpha,\beta)$ can take values $0$, $\pm 1/2$ or $\pm 1$. 
With these conventions, the distance $d(\lambda,\mu)$ between the
weights $\lambda,\mu\in\Lambda$ equals
\begin{itemize}
\item $0$ if $(\lambda,\mu)=1$;
\item $1$ if $(\lambda,\mu)=1/2$;
\item $2$ if $(\lambda,\mu)=0$;
\item $3$ if $(\lambda,\mu)=-1/2$ (and then $\lambda+\mu=\delta$).
\end{itemize}
Thus, for any weight $\lambda\in\Lambda$ there exists a unique
weight at distance $3$; this is the weight $\delta-\lambda$, 
which will be denoted by $\overline\lambda$.

In \cite{Stein_stability, Plotkin_Semenov_Vavilov,
Vavilov_structure, Vavilov_third_look, Vavilov_znaki} 
one can find many further details as to how to recover,
from this diagram alone, the action of
root unipotents $x_{\alpha}(\xi)$, $w_{\alpha}(\eps)$, $h_{\alpha}(\eps)$, 
the signs of structure constants, the shape and signs of equations, 
etc. These and other such similar items are tabulated in 
\cite{Vavilov_Luzgarev_miniE7}. Formally, an explicit knowledge
of these things is not necessary to understand the proofs 
produced in the present paper. However, in reality all
calculations in~\S3, 4, 7--9, were performed with the 
heavy use of weight diagrams, and could hardly be possible 
without them.

Over a field, and in general over a semilocal ring, the extended
Chevalley group $\overline G(\EE_7,R)$ is generated by the usual 
Chevalley group $G(\EE_7,R)$ and the weight elements
$h_{\varpi_7}(\eta)$, $\eta\in R^*$. In the natural numbering
of weight the element $h_{\varpi_7}(\eta^{-1})$ acts on the module
$V(\varpi_7)$ as follows:
\begin{align*}
&\diag(\eta^{-1},1,1,1,1,1,1,1,1,1,1,1,1,1,1,1,1,1,\eta,1,1,\eta,1,1,\eta,1,1,\eta,\\
&\quad\qquad 1,\eta,\eta,1,\eta,\eta,1,\eta,\eta,1,\eta,\eta,\eta,\eta,\eta,\eta,
\eta,\eta,\eta,\eta,\eta,\eta,\eta,\eta,\eta,\eta,\eta,\eta^2).
\end{align*}
Here we assume that the weights are linearly ordered as follows:
$1,\dots,28,\overline{28},\dots,\overline{1}$). Observe, that 
the exponent of $\eta$ increases by 1 each time we cross an
edge marked $\alpha_7$.


\section{The invariants of degree 4}\label{sec:invariants}

In our paper \cite{Vavilov_Luzgarev_norme6} the simply connected
Chevalley group of type $\EE_6$ acting on the $27$-dimensional
module $V=V(\fo1)$ was identified with the isometry group 
of a three-linear form $T\colon V\times V\times V\longrightarrow R$.
There is a similar, but much more complicated description of
the simply connected Chevalley group of type $\EE_7$ acting on the
$56$-dimensional module $V=V(\fo7)$. In this case, to determine
the group one needs to invariants, one of degree 2, and another
one of degree 4. First of all, the module $V$ is self-dual and
carries a unimodular symplectic form $h$. Further, there exists a
four-linear form $f:V\times V\times V\times V\longrightarrow R$ 
such that $G$ can be identified with the full isometry group of
the pair $h$, $f$, in other words, with the group of all
$g\in\GL(V)$ such that $h(gu,gv)=h(u,v)$ and
$f(gu,gv,gw,gz)=f(u,v,w,z)$ for all $u,v,w,z\in V$.
The similarities of this pair of forms define the {\it extended\/} 
Chevalley group $\overline G(\EE_7,R)$ (see Theorem~\ref{thm:g_fh}).

It is obvious how to construct $h$. Construction of the fourth
degree invariant is considerably more complicated, and classically
one constructs not the four-linear form $f$, but rather the 
corresponding quartic\footnote{This form of degree~4 first occurred
in a 1901 paper by L.~E.~Dickson in the context of the 
28 bitangents, and thus, of the Weyl group $W(\EE_7)$. Apparently,
Dickson has not noticed an explicit connection with the group of
type $\EE_7$ itself. Otherwise, Chevalley groups could had been 
discovered some 50 years earlier!}. The fact that the group 
$G$ preserves a form of degree 4 in 56 variables, was first 
observed by E.~Cartan, at least in characteristic~0, but his
explicit construction of this form was flawed (probably, it was 
just a misprint). A very elegant construction of such an invariant
over a field $K$ of characteristic distinct from 2 was given by 
H.~Freudenthal. Namely, he identifies the module $V$ with the
space $A(8,K)^2$, where $A(8,K)$ is the set of antisymmetric 
$8\times 8$ matrices, and considers the following symplectic
inner product and form of degree 4:
\begin{align*}
h((a_1,b_1),(a_2,b_2)) &= \frac12(\tr(a_1b_2^t)- \tr(a_2b_1^t)), 
\\
Q((a,b))&=\ppf(a)+\ppf(b)-\frac14\tr((ab)^2)+\frac1{16}\tr(ab)^2.
\end{align*}
Now, for all characteristics distinct from 2, one can identify
the isometry group of this pair with the simply-connected
Chevalley group $G$ of type $\EE_7$ over $K$ 
(see \cite{Aschbacher_multi, Cooperstein_E7}). The constructions
of the above form in the papers by M.~Aschbacher and B.~Cooperstein
is somewhat different. Actually, in \cite{Aschbacher_multi} the
form is constructed in terms of $\rmA_6$ (the gist of this
construction is expressed by the partition $56=7+21+21+7$), 
whereas the construction in \cite{Cooperstein_E7} is closer to 
Freudenthal's original construction, and is phrased in terms of 
$\rmA_7$ (where $56=28+28$). The isometry group of the form
$Q$ is generated by $G$ and a diagonal element of order 2 
(see \cite{Cooperstein_E7}). 
There are no serious complications in characteristic $p\ge 5$,
whereas characteristic 3 requires some extra-care. 

However, in characteristic 2 this approach is almost immediately 
blocked by serious obstacles. Obviously, the above construction 
fails. Apparently, in characteristic 2 there are whatsoever no 
non-trivial symmetric $G$-invariant four-linear forms on $V$, 
(see~\cite{Aschbacher_multi}). This is related to the fact that
in characteristic 2 the four-linear form
$$
f_0(u,v,x,y)=h(u,v)h(x,y)+h(u,x)h(v,y)+h(u,y)h(v,x),
$$
obtained by the squaring of the symplectic form, becomes
symmetric, which is not the case in characteristics $\ge 3$. 
Actually, in characteristic 2 M.~Aschbacher \cite{Aschbacher_multi} 
constructs a four-linear $G$-invariant form $F$, which is 
symmetric with respect to the {\it even\/} permutations.

There are other constructions of the form $Q$, most notably a
construction by R.~Brown \cite{Brown_groups}, which works in
characteristics $\neq 2,3$. Let $V$ be a space with a non-degenerate
inner product. Then to define a three-linear form on $V$ is
essentially the same as to define on $V$ an algebra structure.
By the same token, to define a four-linear form on $V$ 
is essentially the same as to define on $V$ a {\it ternary\/} 
algebra structure. Indeed, there exists a remarkable ternary
algebra of dimension 56, constructed in terms of the exceptional
27-dimensional Jordan algebra $\mathbb J$ (see~\cite{Brown_groups, Faulkner_Ferrar} 
and references there). This algebra consists of $2\times 2$ matrices
over $\mathbb J$ with scalar diagonal entries, $56=1+27+27+1$.

The orbits of the group $G=G(\EE_7,K)$ on the 56-dimensional 
module are classified in \cite{Haris} in the absolute case, 
in \cite{Liebeck_Seitz_subgroup_structure_1998} for finite fields,
and in \cite{Cooperstein_E7} for arbitrary fields. Basically, these orbits
are described in terms of the four-linear form. Again characteristics
2 and 3 require separate analysis, and to take care of all details
one has to use the notion of 4-forms, introduced by M.~Aschbacher
(see~\cite{Aschbacher_multi, Cooperstein_E7}). Essentially, a
4-form is a form of degree 4 together with all of its 
polarizations. For simplicity, assume that $\Char K\neq 2,3$. Then
a vector $u\in V$ is called {\it singular\/}, if 
$F(u,u,x,y)=0$ for all $x,y\in V$; {\it brilliant\/},
if $F(u,u,u,x)=0$ for all $x\in V$; and {\it luminous\/} if 
$F(u,u,u,u)=0$. Otherwise, i.e.~if $F(u,u,u,u)\neq 0$, the vector
$u$ is called {\it dark\/}. The orbits of the group $G$ on $V$ 
are as follows: 0, non-zero singular vectors, non-singular brilliant
vectors, luminous vectors that are non brilliant, and, finally,
one or several orbits of dark vectors, parametrized by $K^*/K^{*2}$ 
(these last orbits fuse to one orbit under the action of the 
{\it extended\/} Chevalley group 
$\overline G$ of type $\EE_7$).

The feeling that $\EE_7$ stands in the same relation to $\EE_6$, 
as $\EE_6$ itself stands to $\DD_5$, suggests the following 
definition of the form of degree 4 on $V$. Take a base vector
$v^{\lambda}$. Then the vectors $v^{\mu}$,
$d(\lambda,\mu)=2$, generates a 27-dimensional module $U$, 
that supports the cubic form related to $\EE_6$. Let us define
{\it tetrads\/} as quadruples $(\lambda_1,\lambda_2,\lambda_3,\lambda_4)$ 
of pair-wise orthogonal weights. Let $\Theta$ and $\Theta_0$ be the
sets of ordered and unordered tetrads, respectively. Clearly, 
$|\Theta|=56\cdot 27\cdot 10$, whereas $|\Theta_0|=|\Theta|/24=630$. 
Now, we can tentatively define the form $Q_{\tent}$ of degree 4
by setting $Q_{\tent}(x)=\sum\pm 
x_{\lambda_1}x_{\lambda_2}x_{\lambda_3}x_{\lambda_4}$, where the sum
is taken over all $\{\lambda_1,\lambda_2,\lambda_3,\lambda_4\}\in\Theta_0$, 
while the signs are determined by the condition that the resulting form is 
invariant under the action of the extended Weyl group 
$\widetilde W$. Here, one should be slightly more cautious then in the
case of $\EE_6$, since now, in addition to the two possible cases
that occurred there, the following possibility occurs: $w_{\alpha}$ 
moves all 4 weights of a tetrad, two of them in positive and the other
two in negative direction, in which case the signs does not change. 
Nevertheless, an expression of the sign in terms of $h(\lambda_i,\mu_i)$ 
still works. This is essentially the same, as to define the four-linear
form $F_{\tent}$ by
$F_{\tent}(v^{\lambda_1},v^{\lambda_2},v^{\lambda_3},v^{\lambda_4})=
(-1)^{h(\lambda_1,\lambda_2,\lambda_3,\lambda_4)}$, for a tetrad
$(\lambda_1,\lambda_2,\lambda_3,\lambda_4)\in\Theta$ and by
$F_{\tent}(v^{\lambda_1},v^{\lambda_2},v^{\lambda_3},v^{\lambda_4})=0$
otherwise. By construction, this form is invariant under the action 
of $\widetilde W$, and we only have to verify that it is invariant
under the action of the root subgroup $X_{\alpha}$, for some root
$\alpha\in\Phi$. Unfortunately, this is not the case. Namely, for
any tetrad $(\lambda_1,\lambda_2,\lambda_3,\lambda_4)$ and any elementary
root unipotent $g=x_{\alpha}(\xi)$ the following formula holds
$$
F_{\tent}(gv^{\lambda_1},gv^{\lambda_2},gv^{\lambda_3},gv^{\lambda_4})=
F_{\tent}(v^{\lambda_1},v^{\lambda_2},v^{\lambda_3},v^{\lambda_4}).
$$
As it happens, though, there exist quadruples of weights that are not 
tetrads, for which the right hand side equals 0, whereas the left 
hand side is distinct from 0. For instance, take the four weights
$\lambda_1,\lambda_2,\lambda_3,\lambda_4$ such that 
$\lambda_1+\alpha,\lambda_2+\alpha,\break \lambda_3+\alpha,
\lambda_4-\alpha$ are weights, and together the 8 above weights
form a cube (in other words, the corresponding weight diagrams is
the tensor product of three copies $(\rmA_1,\fo1)$, see 
\cite{Cooperstein_E7, Plotkin_Semenov_Vavilov}). 
Then one of the weights $\lambda_1,\lambda_2,\lambda_3$ will be adjacent 
with the other two, say, $d(\lambda_1,\lambda_2)=d(\lambda_1,\lambda_3)=1$, 
so that
$F_{\tent}(v^{\lambda_1},v^{\lambda_2},v^{\lambda_3},v^{\lambda_4})=0$. 

At the same time, decomposing the expression 
$F_{\tent}(gv^{\lambda_1},gv^{\lambda_2},gv^{\lambda_3},gv^{\lambda_4})$ 
by linearity, we get 8 summands, of which exactly one, namely
$F_{\tent}(v^{\lambda_1+\alpha},v^{\lambda_2},v^{\lambda_3},v^{\lambda_4})$,
corresponds to a tetrad, and equals $\pm1$. Thus, the form 
$F_{\tent}$ is not preserved by the action of $X_{\alpha}$.

In itself, this is not yet critical, since one can hope to
save the situation by throwing in another Weyl orbit of monomials.
This is, however, exactly the point where real problems start. 
As a matter of fact, in the above counter-example throwing in 
another orbit of monomials will produce {\it two\/} non-zero
extra summands, so that the resulting correction will be a 
multiple of 2. This means that one cannot define an invariant 
form of degree 4 by setting its values on the tetrads to be 
equal to $\pm1$, one should start with $\pm2$ instead. This is 
precisely where {\it serious\/} trouble starts. In characteristic 
$\neq 2$ the above construction is {\it essentially\/} correct, 
in the sense that it tells how the {\it relevant\/} part of
an invariant form of degree 4 looks like, responsible for the
reduction to $\EE_6$. Let us fix a vector $v^{\lambda}$. Then
$F(v^{\lambda},{-},{-},{-})$ consists of two parts: the form
$F_{\tent}$, as defined above, and another part, introduced
for the resulting form to be $G$-invariant. This second part
has the form $F(v^{\lambda},v^{\lambda^*},{-},{-})$ and does not
say anything beyond the fact that our group preserves 
the usual symplectic form. 

In the works by Jacob Lurie \cite{Lurie} and the second author
\cite{Luzgarev_e7_invariants}, these difficulties arising in
characteristic 2 were sorted out in a systematic way, but the
resulting four-linear forms are not anymore symmetric. Namely, 
let $\mathfrak g$ be the Lie algebra of type $\EE_8$; recall that
$$
\delta = \dynkin23465432
$$
is the highest root of $\EE_8$. The coefficient with which 
$\alpha_8$ occurs in the expansion of a root $\alpha\in\EE_8$ 
with respect to the fundamental roots, is called the $\alpha_8$-height
of $\alpha$ and can only take values 
$-2$, $-1$, $0$, $1$, $2$. This defines the following length 5
grading of the algebra $\mathfrak g$:
$$
\mathfrak g = \mathfrak g_{-2}\oplus \mathfrak g_{-1}\oplus \mathfrak g_0
\oplus \mathfrak g_1 \oplus \mathfrak g_2.
$$
The 56-dimensional space $\mathfrak g_1$ has a base consisting of 
the elementary root elements $e_\alpha$, where $\alpha$ runs over
the roots of $\alpha_8$-height $1$, i.e.~the weights of $V(\varpi_7)$.
Let $\lambda,\mu,\nu,\rho$ be four weights of $V(\varpi_7)$.
Clearly, the element 
$$
[[[[e_{-\delta},e_\lambda],e_\mu],e_\nu],e_\rho]
$$
has $\alpha_8$-weight $2$, so that the resulting element is a multiple 
of $e_\delta$. Denote the corresponding scalar coefficient by 
$c(\lambda,\mu,\nu,\rho)$ and consider the four-linear form 
$$
f(u,v,w,z) = \sum_{\lambda,\mu,\nu,\rho\in\Lambda}c(\lambda,\mu,\nu,\rho)
u_\lambda v_\mu w_\nu z_\rho.
$$
Obviously, this form is invariant under the action of the group
$G(\EE_7,R)$ on the module $V(\varpi_7)$.

\subsection*{The orbit of the highest weight vector} It is well
known that in any representation of the group $G$ the orbit 
$Gv^+$ of the highest weight vector $v^+$ is an intersection of
quadrics~\cite{Lichtenstein}. Here, as a motivation for the next
section, we explicitly describe the equations defining the orbit
of $v^+$ for the microweight representation of 
$\EE_7$. For the microweight representation of $\EE_6$ this was
done in~{\rm [11]}. Of course, for these cases the corresponding 
equations were found by H.~Freudenthal and J.~Tits more than 40
years ago (see also \cite{Vavilov_Luzgarev_miniE7} and references
there), but again we wish to show how to recover the equations
directly from the weight diagram.

Let $\omega=\fo1$ for $\EE_6$ or $\omega=\fo7$ for
$\EE_7$, the case of $(\Phi,\omega)=(\EE_6,\fo6)$ is dual to the
first case. We use the same interpretation of the modules as 
in \S~1. In particular, 
$\Phi=\EE_l$, $l=6,7$, $\Delta=E_{l+1}$, and
$\Sigma=\Sigma_{l+1}(1)$. The group $G=G(\Phi,R)$ acts on
$V=U_{l+1}(1)/U_{l+1}(2)\cong\prod X_{\alpha}$,
$\alpha\in\Sigma$ by conjugation. Since we are only interested in
the equations satisfied by the orbit $Gv^+$, we can assume that
$R=K$ is an (algebraically closed) field\footnote{For rings, there
are further obstacles, related to the fact that the lower
$K$-functors, or their analogues, can be non-trivial, that we
do not discuss here.}.

In both cases one can take $v^{+}=v^{\rho}=x_{\rho}(1)$ as the
highest weight vector, where
$$
\rho=\dynkin2234321{}
\quad\text{or}\quad
\rho=\dynkin23465431
$$
is the maximal root of $\EE_7$, or the unique submaximal
root of $\EE_8$, respectively. Recall that the vector 
$a=(a_{\alpha})\in V$ is now viewed as the product
$x=\prod x_{\alpha}(a_\alpha)\in\prod X_{\alpha}$,
$\alpha\in\Sigma$. In the case of $\EE_7$ this product is
considered modulo $U_8(2)=X_{\rho+\alpha_8}$, the
root subgroup corresponding to the maximal root of
$\EE_8$.


\section{Construction of the system of quadrics}\label{sec:construction}

The first set of quadrics defining the highest weight orbit, 
consists of {\em square equations}; for large classes of representations,
their construction and numerology were described by the first author in
\cite{Vavilov_numerology, Vavilov_more_numerology}.
Here, we recall some basic definitions of \cite{Vavilov_numerology} in
the context of the 56-dimensional representation of $(\EE_7,\varpi_7)$.

The set of weights $\Omega\subseteq\Lambda$ is called a 
{\it square\/} if $|\Omega|\geq 4$ and for all $\lambda\in\Omega$
its difference $\lambda-\mu$ with all weights $\mu\in\Omega$, except
exactly one, denoted by $\lambda^*$, is a root, whereas the
difference $\lambda-\lambda^*$ is not a root (and thus
$\lambda\perp\lambda^*$). A square maximal with respect to inclusion
is called a {\it maximal square\/}. In \cite{Vavilov_numerology} it
is proven that for our representation each maximal square $\Omega$
consists of $12$ weights and the sum $\lambda+\lambda^*$ does not
depend on the choice of $\lambda\in\Omega$. At that, a maximal square 
$\Omega$ is completely determined by this sum. Furthermore, in the
case of a microweight representation of $\EE_7$ maximal squares are
in bijective correspondence with the roots of $\EE_7$, namely, to a
root $\alpha\in\EE_7$ there corresponds the square
$$
\Omega(\alpha) = \{\lambda\in\Lambda\mid\lambda-\alpha\in\Lambda\}.
$$

Let, as above, $\Omega$ be some maximal square. Choose orthogonal
weights $\rho,\rho^*\in\Omega$ and define the polynomial
$f_{\rho,\rho^*}\in\Int[\{x_\lambda\}_{\lambda\in\Lambda}]$
by
$$
f_{\rho,\rho^*} = x_\rho x_{\rho^*} - \sum N_{\rho,-\lambda}
N_{\rho^*,-\lambda^*} x_\lambda x_{\lambda^*},
$$
where the sum is taken over all orthogonal pairs of weights
$\{\lambda,\lambda^*\}$, except $\{\rho,\rho^*\}$ itself.
The equation $f_{\rho,\rho^*}(v)=0$ on the components of
a vector $v =(v_\lambda)_{\lambda\in\Lambda}\in V$ was called
in \cite{Vavilov_numerology} a {\it square equation\/},
corresponding to the maximal square $\Omega$. In particular,
this equation depends only on the square $\Omega$ itself, and 
not on the arbitrary choice of a pair $\rho,\rho^*$ of orthogonal
weights: passing to another such pair, the polynomial 
$f_{\rho,\rho^*}$ is multiplied by $\pm 1$.

Fixing in each maximal square $\Omega$ one such pair of orthogonal
weights gives us $126$ polynomials, corresponding (up to sign)
to the $126$ maximal squares (or, what is the same, to the $126$ 
roots of $\EE_7$).

Finally, for a root $\alpha\in\EE_7$ we consider the polynomial 
$g_\alpha\in\Int[\{x_\lambda\}_{\lambda\in\Lambda}]$, defined by
$$
g_\alpha = \sum_{\lambda\in\Omega(\alpha)}N_{\lambda,\overline\lambda}x_\lambda 
x_{\overline\lambda}.
$$
Again, by definition these polynomials are in bijective correspondence 
with the roots of $\EE_7$ (i.e. there are $126$ of them), but the
following lemma asserts that it suffices to consider only $g_\alpha$
corresponding to $\alpha\in\Pi$.

\begin{lemma}
The ideal in $\Int[\{x_\lambda\}_{\lambda\in\Lambda}]$, generated by
the
polynomials $\{g_\alpha\}$, $\alpha\in\EE_7$, coincides with the
ideal, generated by the polynomials $\{g_\alpha\}_{\alpha\in\Pi}$.
\end{lemma}

\begin{proof}
Let $\alpha\in\EE_7$.
Observe, that 
$\Omega(\alpha) = \{\lambda\in\Lambda\mid\lambda-\alpha\in\Lambda\}
= \{\lambda\in\Lambda\mid (\lambda,\alpha)=1/2\}$.
If $\lambda\in\Omega(\alpha)$,
then $(\overline\lambda,-\alpha)=(\delta-\lambda,-\alpha) = (\lambda,\alpha)-(\delta,\alpha) = (\lambda,\alpha)$,
since $\delta\perp\alpha$ for all $\alpha\in\EE_7$.
Thus, $\Omega(-\alpha) = \{\overline\lambda\mid\lambda\in\Omega(\alpha)\}$.
We get that 
\begin{equation*}
g_{-\alpha} = \sum_{\lambda\in\Omega(-\alpha)} N_{\lambda,\overline\lambda}x_\lambda x_{\overline\lambda}
= \sum_{\lambda\in\Omega(\alpha)} N_{\overline\lambda,\lambda}x_\lambda x_{\overline{\lambda}} = - g_\alpha.
\end{equation*}

Now, let $\alpha,\beta,\alpha+\beta\in\EE_7$. Let us show that $g_{\alpha+\beta} =
g_\alpha + g_\beta$.
Observe, that $(\lambda,\alpha+\beta) = (\lambda,\alpha)+(\lambda,\beta)$,
and each of these inner products equals $0$ or $\pm 1/2$.
If $\lambda\in \Omega(\alpha+\beta)$, i.e.\ $(\lambda,\alpha+\beta)=1/2$,
then one of the expressions $(\lambda,\alpha)$, $(\lambda,\beta)$ equals 
$1/2$, while the other one is 0. In this case the monomial
$N_{\lambda,\overline\lambda}x_\lambda x_{\overline{\lambda}}$ is contained
either in $g_\alpha$, or in $g_\beta$, but not in both. Conversely,
if $\lambda\in\Omega(\alpha)$ and at that $\lambda\notin\Omega(\alpha+\beta)$, 
then necessarily $(\lambda,\beta)=-1/2$, which implies that 
$(\overline\lambda,\beta)=1/2$. It follows that $g_\alpha$ contains the
monomial 
$N_{\lambda,\overline\lambda}x_\lambda x_{\overline\lambda}$, while 
$g_\beta$ contains the monomial
$N_{\overline\lambda,\lambda}x_\lambda x_{\overline\lambda}=
-N_{\lambda,\overline\lambda}x_\lambda x_{\overline\lambda}$, which cancel.
\end{proof}

Set $g_i=g_{\alpha_i}$ and let $I$ be the ideal generated by the above
quadratic polynomials 
$f_{\mu,\mu^*}$ and the polynomials
$g_i$, $i=1,\dots,7$ (altogether, this gives us $126+7=133$ polynomials).

\begin{theorem}\label{thm:e7_in_fix}
Denote by $\Fix_R(I)$ the set of $R$-linear transformations,
preserving the ideal $I${\rm:}
$$
\Fix_R(I) = \{g\in\GL(56,R)\mid f(gx)\in I\text{ for all }f\in I\}.
$$
Then the elementary Chevalley group $E(\EE_7,R)$ is contained in $\Fix_R(I)$.
\end{theorem}


\section{Proof of Theorem~\ref{thm:e7_in_fix}}

Since we realize the representation of the group of type $\EE_7$ 
inside the group $\EE_8$, the calculations of this section mostly reproduce
the calculations in \cite{Luzgarev2014}, for the adjoint representation
of $\EE_8$. However, we cannot directly cite the results of \cite{Luzgarev2014},
since we are only interested on the parts of these polynomials that 
correspond to the weights of $\EE_7$ inside $\EE_8$.

To prove Theorem \ref{thm:e7_in_fix} it suffices to show that if
$g=x_\gamma(\xi)$ (here $\gamma\in\EE_7$, $\xi\in R$), then
$f_{\rho,\sigma}(gx)\in I$ for all $\rho,\sigma\in\Lambda$,
$\rho\perp\sigma$, and $g_\alpha(gx)\in I$ for all $\alpha\in\Phi$.

The following special case of Matsumoto lemma
(see.~\cite[Lemma~2.3]{Matsumoto}), describes the action of an
elementary root unipotent $x_\gamma(\xi)$ on the vectors in $V$.

\begin{lemma}\label{Matsumoto}
Let $\gamma\in\EE_7$, $v\in V$.
\begin{enumerate}
\item 
If $\lambda\in\Lambda$, $\lambda-\gamma\notin\Lambda$, then
  $(x_\gamma(\xi)v)_\lambda=v_\lambda$.
\item 
If $\lambda,\lambda-\gamma\in\Lambda$, then
$(x_\gamma(\xi)v)_\lambda=v_\lambda+N_{\gamma,\lambda-\gamma}\xi v_{\lambda-\gamma}$.
\end{enumerate}
In particular, if $(\gamma,\lambda)\neq 1/2$, 
then $(x_\gamma(\xi)v)_\lambda=v_\lambda$.
\end{lemma}

Lemma~\ref{Matsumoto} was stated in terms of the structure constants
$N_{\alpha\beta}$ of the Lie algebras of type $\EE_8$. It is
classically known
(see references in \cite{Vavilov_znaki}), that they satisfy the
following relations:
\begin{align*}
&N_{\alpha\beta}=N_{-\beta,-\alpha} = -N_{-\alpha,-\beta} = -N_{\beta\alpha},
\\
&N_{\alpha\beta} = N_{\beta\gamma} = N_{\gamma\alpha}.
\end{align*}
Moreover, they are subject to the {\em cocycle identity}:
\begin{align*}
N_{\beta\gamma}N_{\alpha,\beta+\gamma} = N_{\alpha+\beta,\gamma}N_{\alpha\beta}.
\end{align*}
In the sequel we use these equalities without any specific reference.

Recall that $f_{\rho,\sigma}$ corresponds to the maximal square 
$\Omega = \Omega(\alpha)$, for some root $\alpha\in\EE_7$. The
weights of the square $\Omega$ can be partitioned into pairs of
orthogonal weights $\{\lambda,\mu\}$ with
$\lambda+\mu = \delta+\alpha$; the polynomial $f_{\rho,\sigma}$ 
consists of monomials of the form $\pm x_\lambda x_\mu$ for all
such pairs. Let us trace what happens with such monomials when
$x$ is mapped to $gx = x_\gamma(\xi)x$. Observe, that 
$(\lambda,\gamma)$ can only take values 
$0$ and $\pm 1/2$. If $(\lambda,\gamma)\neq 1/2$, the by 
Lemma~\ref{Matsumoto} one has $(gx)_\lambda = x_\lambda$.
Calculating the inner product of $\lambda+\mu = \delta+\alpha$ 
by $\gamma$ and recalling that $(\delta,\gamma) = 0$, we get that
$$
(\lambda,\gamma) + (\mu,\gamma) = (\alpha,\gamma).
$$
The following cases can possibly occur:
\begin{itemize}
\item 
$(\alpha,\gamma) = -1$ or $-1/2$. Then none of the summands on
the right hand side equals $1/2$, and thus $f_{\rho,\sigma}(gx) =
  f_{\rho,\sigma}(x)\in I$.
\item 
$(\alpha,\gamma) = 1/2$. Then exactly one summand on the right 
hand side equals $1/2$, whereas the second one is $0$. Let,
for instance, $(\lambda,\gamma) = 1/2$ and $(\mu,\gamma) =
  0$. Similarly, suppose that $(\rho,\gamma) = 1/2$ and
  $(\sigma,\gamma) = 0$. Then
  $(gx)_\lambda (gx)_\mu = (x_\lambda +
  \xi N_{\gamma,\lambda-\gamma}x_{\lambda-\gamma})x_\mu$.
  Expanding these equalities, summing with signs over all pairs 
  of orthogonal weights, we get 
$$
  f_{\rho,\sigma}(gx) = f_{\rho,\sigma}(x) +
  \xi N_{\gamma,\rho-\gamma}x_{\rho-\gamma}x_\sigma -
  \xi\sum N_{\rho,-\lambda}N_{\sigma,-\mu}N_{\gamma,\lambda-\gamma}
  x_{\lambda-\gamma} x_\mu.
$$
  Observe, that the sum $\lambda-\gamma+\mu = \delta+\alpha-\gamma$ 
  does not depend on $\lambda$. Thus, the pairs of weights 
  $\{\rho-\gamma,\sigma\}$ and
  $\{\lambda-\gamma,\mu\}$ appear in the maximal square 
  $\Omega(\alpha-\gamma)$.  Let us verify that
$$
  f_{\rho,\sigma}(gx) = f_{\rho,\sigma}(x) + \xi N_{\gamma,\rho-\gamma}
  f_{\rho-\gamma,\sigma}(x).
$$
With this end, it only remains to check that the signs coincide:
$$
  N_{\rho,-\lambda}N_{\sigma,-\mu}N_{\gamma,\lambda-\gamma} =
  N_{\gamma,\rho-\gamma}N_{\rho-\gamma,\gamma-\lambda}N_{\sigma,-\mu}.
$$
But this immediately follows from the cocycle identity.
\item
$(\alpha,\gamma) = 0$. Then either both summand on the left hand
side are $0$, or one of them equals $1/2$, while the other one equals 
$-1/2$. The summands, for which $(\lambda,\gamma) = (\mu,\gamma) = 0$, 
do not contribute to the difference 
$f_{\rho,\sigma}(gx) - f_{\rho,\sigma}(x)$. Now, let
  $(\lambda,\gamma) = 1/2$ and $(\mu,\gamma) = -1/2$. Then
  $\lambda-\gamma$ and $\mu+\gamma$ are weights that sum to
  $\lambda+\mu=\delta+\alpha$; Moreover, $(\lambda-\gamma,\gamma) =
  -1/2$ and $(\mu+\gamma,\gamma) = 1/2$. Thus, the weights
  $\{\lambda-\gamma,\mu+\gamma\}$ are orthogonal and, thus, belong to 
  the same maximal square $\Omega$. At that, 
\begin{align*}
&N_{\rho,-\lambda}N_{\sigma,-\mu}(gx)_\lambda (gx)_\mu 
  + N_{\rho,-\lambda+\gamma}N_{\sigma,-\mu-\gamma}
  (gx)_{\lambda-\gamma}(gx)_{\mu+\gamma}
\\
&\quad = N_{\rho,-\lambda}N_{\sigma,-\mu}x_\lambda x_\mu 
  + N_{\rho,-\lambda+\gamma}N_{\sigma,-\mu-\gamma}
  x_{\lambda-\gamma}x_{\mu+\gamma}
\\
&\qquad+ \xi N_{\rho,-\lambda}N_{\sigma,-\mu}N_{\gamma,\lambda-\gamma}
  x_{\lambda-\gamma}x_\mu
  + \xi N_{\rho,-\lambda+\gamma}N_{\sigma,-\mu-\gamma}
  N_{\gamma,\mu}x_{\lambda-\gamma}x_\mu.
\end{align*}
An easy calculation shows that the summands, containing $\xi$,
cancel. Thus, $f_{\rho,\sigma}(gx) = f_{\rho,\sigma}(x)$.
\item 
$(\alpha,\gamma) = 1$, i.e. $\alpha=\gamma$.
  In this case $(\lambda,\gamma) = (\mu,\gamma) = 1/2$.
  Thus, 
\begin{align*}
(gx)_\lambda (gx)_\mu &= 
  (x_\lambda+\xi N_{\gamma,\lambda-\gamma}x_{\lambda-\gamma})
  (x_\mu + \xi N_{\gamma,\mu-\gamma}x_{\mu-\gamma})
\\
&= x_\lambda x_\mu
  + \xi N_{\gamma,\lambda-\gamma} x_{\lambda-\gamma}x_\mu 
\\
&\quad + \xi N_{\gamma,\mu-\gamma} x_\lambda x_{\mu-\gamma}
  + \xi^2 N_{\gamma,\lambda-\gamma}N_{\gamma,\mu-\gamma}
    x_{\lambda-\gamma}x_{\mu-\gamma}.
\end{align*}
Altogether, we get $6$ summands containing $\xi^2$; the corresponding
monomials are of the form $\pm x_{\lambda-\gamma}x_{\mu-\gamma}$, where
  $(\lambda-\gamma) + (\mu-\gamma) = \lambda + \mu - 2\gamma = \delta
  - \alpha$. Thus, the $6$ pairs of weights of the form
  $\{\lambda-\gamma,\mu-\gamma\}$ constitute the maximal square
  $\Omega(-\alpha)$.
  It remains to verify that also the signs of these summands 
  coincide with the signs in the square equation corresponding
  to $\Omega(-\alpha)$. Indeed,
$$
  N_{\rho,-\lambda}N_{\sigma,-\mu} N_{\gamma,\lambda-\gamma}N_{\gamma,\mu-\gamma}
  = N_{\gamma,\rho-\gamma}N_{\gamma,\sigma-\gamma}
  N_{\rho-\gamma,-\lambda+\gamma}N_{\sigma-\gamma,-\mu+\gamma}.
$$
Finally, $12$ summands in the above sum contain $\xi$; the corresponding 
monomials are of the form $\pm x_{\lambda-\gamma}x_\mu$ and
  $\pm x_\lambda x_{\mu-\gamma}$, where $(\lambda-\gamma)+\mu =
  \lambda + (\mu-\gamma) = \delta$. It is easy to see that these are
  precisely the monomials that occur in $g_\alpha$. It only remains to
  verify that their signs agree:
$$
N_{\rho,-\lambda}N_{\sigma,-\mu}N_{\gamma,\mu-\gamma} = N_{\lambda,\mu-\gamma}.
$$
By summarizing the above, we get
$$
  f_{\rho,\sigma}(gx) = f_{\rho,\sigma}(x)
  + \xi N_{\gamma,\sigma-\gamma}N_{\rho,\sigma-\gamma} g_\alpha(x)
  + \xi^2  N_{\gamma,\rho-\gamma}N_{\gamma,\sigma-\gamma}
    f_{\rho-\gamma,\sigma-\gamma}.
$$
\end{itemize}

Next, we look at $g_\alpha(gx)$. The monomials that occur in
$g_\alpha$, are of the form $x_\lambda x_{\overline\lambda}$, where 
$\lambda$ runs over the maximal square $\Omega(\alpha)$, whereas 
$\overline\lambda = \delta-\lambda$.
Taking the inner product of $\lambda + \overline\lambda = \delta$ 
with $\gamma$, we get:
$$
(\lambda,\gamma) + (\overline\lambda,\gamma) = (\delta,\gamma) = 0.
$$
Observe, that $(\lambda,\alpha)=1/2$ and $(\overline\lambda,\alpha) = -1/2$.
The inner product $(\alpha,\gamma)$ can take the following values:
\begin{itemize}
\item 
$(\alpha,\gamma) = -1$, i.e. $\gamma=-\alpha$. But then
  $(\lambda,\gamma) = -1/2$, and thus $(gx)_\lambda = x_\lambda$ for
  all $\lambda\in\Omega(\alpha)$. It follows that $g_\alpha(gx) = g_\alpha(x)$.
\item $(\alpha,\gamma) = 1$, i.e. $\gamma=\alpha$.
Then $(gx)_\lambda (gx)_{\overline\lambda} =
(x_\lambda + \xi N_{\gamma,\lambda-\gamma}x_{\lambda-\gamma})
x_{\overline\lambda}$, and thus
$$
g_\alpha(gx) = g_\alpha(x) + \sum_{\lambda\in\Omega(\alpha)}
\xi N_{\lambda,\overline\lambda}N_{\gamma,\lambda-\gamma}x_{\lambda-\gamma}x_{\overline\lambda}.
$$
At that $(\lambda-\gamma) + \overline\lambda = \delta-\gamma$.
Observe that if $\lambda\in\Omega(\alpha)$, then
$\lambda^* = \overline\lambda+\gamma\in\Omega(\alpha)$; thus the weights
$\{\lambda-\gamma,\overline\lambda\}$ form an orthogonal pair of weights
and sit in $\Omega(-\alpha)$. Take an arbitrary $\lambda_0\in\Omega(\alpha)$ 
and set
$\rho = \lambda_0-\gamma$, $\sigma = \overline\lambda_0$.
The equality $N_{\lambda,\overline\lambda}N_{\gamma,\lambda-\gamma} =
N_{\overline\lambda+\gamma,\lambda-\gamma}N_{\gamma,\overline\lambda}$ 
implies that the same on the right hand side equals $2\xi
N_{\rho+\gamma,\sigma}N_{\gamma,\rho} f_{\rho,\sigma}(x)$.
One can conclude that
$g_\alpha(gx) = g_\alpha(x) + 2\xi
N_{\rho+\gamma,\sigma}N_{\gamma,\rho} f_{\rho,\sigma}(x)$.

\item 
$(\alpha,\gamma) = 0$.
Let $\lambda\in\Omega(\alpha)$, i.e. $(\lambda,\alpha) = 1/2$.
If $(\lambda,\gamma) = 1/2$, then $\lambda-\gamma$ is a weight
Moreover, $(\lambda-\gamma,\alpha) = 1/2$, and thus
$\lambda-\gamma\in\Omega(\alpha)$. Furthermore,
$(\lambda-\gamma,\gamma) = -1/2$, and thus
$(\overline{\lambda-\gamma},\gamma) = 1/2$.

Let us look at what happens with the monomials in $g_\alpha$
corresponding to the weights $\lambda, \lambda-\gamma$:
\begin{align*}
& N_{\lambda,\overline\lambda}(gx)_\lambda(gx)_{\overline\lambda}
+ N_{\lambda-\gamma,\overline{\lambda-\gamma}}
(gx)_{\lambda-\gamma}(gx)_{\overline{\lambda-\gamma}}\\
&\quad= N_{\lambda,\overline\lambda}x_\lambda x_{\overline\lambda}
+ \xi N_{\lambda,\overline\lambda}N_{\gamma,\lambda-\gamma}
      x_{\lambda-\gamma}x_{\overline\lambda}\\
&\qquad+ N_{\lambda-\gamma,\overline{\lambda-\gamma}}
x_{\lambda-\gamma}x_{\overline{\lambda-\gamma}}
+ \xi N_{\lambda-\gamma,\overline{\lambda-\gamma}}N_{\gamma,\overline\lambda}
x_{\lambda-\gamma}x_{\overline\lambda}.
\end{align*}
But $N_{\lambda,\overline\lambda}N_{\gamma,\lambda-\gamma} = -
N_{\lambda-\gamma,\overline{\lambda-\gamma}}N_{\gamma,\overline\lambda}$,
so that the summands containing $\xi$ cancel. This shows that 
$g_\alpha(gx) = g_\alpha(x)$.
\item 
$(\alpha,\gamma) = 1/2$.
Let $\lambda\in\Omega(\alpha)$.
Look at the weight $\lambda^* = \delta+\alpha-\lambda\in\Omega(\alpha)$.
One has $\lambda +\lambda^* = \delta+\alpha$, and thus
$(\lambda,\gamma) + (\lambda^*,\gamma) = (\alpha,\gamma) = 1/2$.
This means that one of the summands on the right hand side equals
$1/2$, while another one equals $0$. If follows, that for 6 out of the 12
weights $\lambda\in\Omega(\alpha)$ one has $(\lambda,\gamma) = 1/2$. 
Denote the set of these weights by $L$. It follows that
\begin{align*}
g_\alpha(gx) &= \sum N_{\lambda,\overline\lambda} (gx)_\lambda (gx)_{\overline\lambda}
\\
&= \sum_{\lambda\in\Omega(\alpha)} N_{\lambda,\overline\lambda} x_\lambda x_{\overline\lambda}
+ \sum_{\lambda\in L} \xi N_{\lambda,\overline\lambda}N_{\gamma,\lambda-\gamma}
x_{\lambda-\gamma}x_{\overline\lambda}.
\end{align*}
Observe that the sum $(\lambda-\gamma)+\overline\lambda = \delta-\gamma$ 
does not depend on $\lambda$. This means that the pairs of orthogonal weights
$\{\lambda-\gamma,\overline\lambda\}$ in the second sum belong to the maximal
square $\Omega(-\gamma)$, and since there are 6 such pairs, they exhaust this
square. Let us fix one such pair $\{\rho,\sigma\}$ and show that up to sign
the second sum equals $\xi f_{\rho,\sigma}(x)$. With this end it remains
to notice that
$N_{\lambda,\overline\lambda}N_{\gamma,\lambda-\gamma} =
N_{\rho+\gamma,\sigma}N_{\gamma,\rho}N_{\rho,-\lambda+\gamma}N_{\sigma,-\overline\lambda}$.
Finally, we get
$$
g_\alpha(gx) = g_\alpha(x) + \xi N_{\rho+\gamma,\sigma}N_{\gamma,\rho}
f_{\rho,\sigma}(x).
$$
\item $(\alpha,\gamma) = -1/2$. Observe that $g_{-\alpha} =
  -g_\alpha$, and thus, replacing $\alpha$ by
  $-\alpha$, we fall into the above case. 
\end{itemize}


\section{Proof of Theorem~\ref{thm:g_i}: An outline}

First, let $f_1,\ldots,f_s$ be arbitrary polynomials in $t$ 
variables with coefficients in a commutative ring
$R$ (in the majority of the real world applications, $R=\mathbb Z$ 
or $R={\mathbb Z}[1/2]$). We are interested in the linear changes
of variables $g\in\GL(t,R)$ that preserve the condition that all
these polynomials simultaneously vanish. In other words, we
consider all $g\in\GL(t,R)$ preserving the ideal $A$ of the
ring $R[x_1,\ldots,x_t]$ generated by $f_1,\ldots,f_s$. This last
condition means that for any polynomial $f\in A$ the polynomial 
$f\circ g$ obtained from $f$ by the linear substitution $g$ is
again in $A$. It is well known (see, e.g.\ 
\cite[Lemma~1]{Dixon} or \cite[Proposition~1.4.1]{Waterhouse_det}),
that the set $G_A(R)=\Fix_R(A)=\Fix_R(f_1,\ldots,f_s)$ of all
such linear variable changes $g$ forms a group. For any $R$-algebra
$S$ with 1 we can consider $f_1,\ldots,f_s$ as polynomials with 
coefficients in $S$ and, thus, the group $G(S)$ is defined for all
$R$-algebras. It is clear that $G(S)$ depends functorially on
$S$. It is easy to provide examples showing that
$S\mapsto G(S)$ may fail to be an affine group scheme over $R$. 
This is due to the fact that $G_A(R)$ is defined by {\it
congruences\/}, rather than equations, in its matrix entries.
However, in~\cite{Waterhouse_det}, Theorem~1.4.3 and further, 
a simple sufficient condition was found, that guarantees that
$S\mapsto G(S)$ is an affine group scheme. Denote by
$R[x_1,\ldots,x_t]_r$ the submodule of polynomials of degree
at most $r$. For our purposes it suffices to invoke 
Corollary 1.4.6 of \cite{Waterhouse_det}, pertaining to the case
where $R=\mathbb Z$.

\begin{lemma}
Let $f_1,\ldots,f_s\in {\mathbb
Z}[x_1,\ldots,x_t]$ be polynomials of degree at most 
$r$ and let $A$ be the ideal they generate. Then for the functor
$S\mapsto\Fix_S(f_1,\ldots,f_s)$ to be an affine group scheme,
it suffices that the rank of the intersection $A\cap
R[x_1,\ldots,x_t]_r$ does not change under reduction modulo
any prime $p\in\mathbb Z$.
\end{lemma}

We apply this lemma to the case of the ideal $A=I$ in ${\mathbb Z}[x_{\lambda}]$,
constructed in~\S\ref{sec:construction}. For any commutative ring
$R$ we set $G_I(R)=\Fix_R(I)$.

\begin{lemma}
The functor $R\mapsto G_I(R)$ is an affine group scheme
defined over $\Int$.
\end{lemma}

\begin{proof}
Let us show that for any prime $p$ the $133$ generating the
ideal $I$ are independent modulo $p$. Indeed, specializing 
$x_\lambda$ appropriately, we can guarantee that one of these 
polynomials takes value $1$, while all other vanish. Observe,
that the polynomials $f_{\lambda,\mu}$ only contain monomials
$x_\nu x_\rho$ for $d(\nu,\rho)=2$ and $\nu+\rho=\lambda+\mu$, 
and that for all 126 polynomials of our generating set the
sum $\lambda+\mu$ takes distinct values. Furthermore, the
polynomials 
$g_\alpha$ only contain monomials $x_\nu x_\rho$ for $d(\nu,\rho) =
3$.
Thus, for $f_{\lambda,\mu}$ is suffices to set $x_\lambda=x_\mu=1$ and
$x_\nu=0$ for all other weights, the monomial $x_\lambda x_\mu$ only
occurs in $f_{\lambda,\mu}$. Finally, for $g_i$, $i=1,\dots,7$,
one can set $x_{\lambda_i} = x_{\overline{\lambda_i}} = 1$ and 
$x_\nu=0$ for all other $\nu$, where $\lambda_i$ has the following
property: $\alpha_i$ is the unique fundamental root such that
the difference $\lambda_i - \alpha_i$ is a weight.
As $\lambda_1,\dots,\lambda_7$ one can take, for instance, the
weights $\overline{8}, 8, 6, 4, 3, 2, 1$ (Figure~2).
\begin{figure}[h]
\begin{center}
\includegraphics{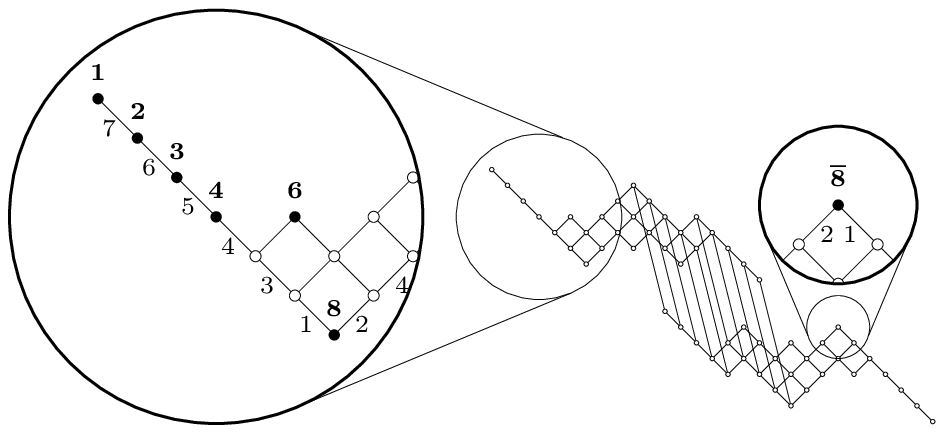}
\caption{}
\end{center}
\end{figure}
\end{proof}

To prove the main results of the present paper, we need to recall some 
further well known facts. The following lemma is Theorem~1.6.1 
of~\cite{Waterhouse_det}.

\begin{lemma}\label{lem:Waterhouse}
Let $G$ and $H$ be affine group schemes of finite type over 
$\Int$, where $G$ is flat, and let $\phi\colon G\to H$ be a morphism
of group schemes. Assume that the following conditions are satisfied
for an algebraically closed field $K${\rm:}
\begin{enumerate}
\item 
$\dim(G_K)\geq\dim_K(\Lie(H_K))$,
\item 
$\phi$ induces monomorphisms on the groups of points $G(K)\to H(K)$ and
  $G(K[\delta])\to H(K[\delta])$,
\item 
the normalizer of $\phi(G^0(K))$ in $H(K)$ is contained in $\phi(G(K))$.
\end{enumerate}
Then $\phi$ is an isomorphism of group schemes over $\Int$.
\end{lemma}

Here $G_0$ denotes the connected component of the identity in $G$,
$G_K$ denotes the scheme obtained from $G$ by a change of scalars,
and $\Lie(H_K)$ denotes the Lie algebra of the scheme $H_K$. Recall
that $K[\delta] = K[x]/(x^2)$ is the algebra of dual numbers over $K$.

Observe, that in our case the preliminary assumptions on the schemes
are satisfied automatically. All schemes considered are of finite type, 
being subschemes of appropriate $\GL_n$. Flatness follows from the fact 
that $G$ is connected and after the change of base to an algebraically
closed field, we will get smooth schemes of the same dimension. 
Thus, we only have to verify the three conditions of the above lemma.


\section{The case of an algebraically closed field}

The following lemma summarizes obvious properties of the
minimal representation $\pi\colon\overline{G}(\EE_7,-) \to \GL_{56}$
of the simply connected Chevalley group of type $\EE_7$.
The fact that $\pi(\overline{G}(\EE_7,-))$ is irreducible and tensor 
indecomposable immediately follows from the fact that $\pi$ is
microweight. Its faithfulness follows from the equality of
weight lattices $\Lambda(\pi)=P(\Phi)$. The claim about
normalizers follows from the classical description of 
abstract automorphisms of Chevalley groups over fields 
(see~\cite{Stein_generators}).
Recall that this description asserts that any algebraic 
automorphism of an extended Chevalley group is the product of an inner
automorphism, a central automorphism and a graph automorphism.
The ordinary Chevalley group may have diagonal automorphisms,
but they become inner in the extended group. Modulo the
algebraic ones, the only 
non-algebraic automorphisms are field automorphisms.
Clearly, groups of type $\EE_7$ do not have any non-trivial
graph automorphisms.

\begin{lemma}\label{lem:irreducible}
Viewed as a subgroup of $\GL(56,K)$, the algebraic group 
$\overline G(\EE_7,K)$ is irreducible and tensor indecomposable.
Moreover, it is equal to its own normalizer. 
\end{lemma}

Let us recall the general outline of the proof of the 
following lemma. It is {\it almost\/} the same as the proof of
Lemma 10 in \cite{Vavilov_Luzgarev_norme6}, but there is a minor
difference, due to the fact that now the extended Chevalley group
of type $\EE_7$ is not maximal in $\GL(n,K)$, but is contained
in the general symplectic group $\GSp(56,K)$. In a classical
1952 paper, Eugene Dynkin~\cite{Dynkin} described the maximal
connected closed subgroups of simple algebraic groups over an
algebraically closed field of characteristic 0. More
precisely, he reduced their description to the representation 
theory of simple algebraic groups. Relying on earlier results by 
Seitz himself, and by Donna Testerman, Gary Seitz~\cite{Seitz_maximal_classical}
generalized this description to subgroups of classical
algebraic groups over an arbitrary algebraically closed field.
Theorem~2 of~\cite{Seitz_maximal_classical} can be stated as 
follows. Let $V$ be the vector representation of $\SL(V)$, 
and $X$ be a proper simple algebraic subgroup of $\SL(n,K)$
such that the restriction $V|X$ of the module $V$ to $X$ is
irreducible and tensor indecomposable. Further, let $Y$ be a 
proper connected closed subgroup of $\SL(n,K)$, that strictly
contains $X$. Then either $Y=\Sp(V)$ or $Y=\SO(V)$, or else the
pair $(X,Y)$ is explicitly listed in
\cite[Table~1]{Seitz_maximal_classical}.

\begin{lemma}\label{lemma:g_i_field}
Theorem~$\ref{thm:g_i}$ holds for any algebraically closed field. 
\end{lemma}

\begin{proof}
It suffices to prove that the connected components of the groups
in question coincide. Since $\overline G(\EE_7,K)$ coincides with
its normalizer, it will automatically follow that the group 
$G_I(K)$ is connected. The fact that $\overline G(\EE_7,K)$ 
stabilizes the requisite system of forms, follows from 
Theorem~\ref{thm:e7_in_fix}. The inverse inclusion can be 
established as follows. In Table~1 of~\cite{Seitz_maximal_classical} 
the group of type $\EE_7$ occurs in the column {\tt X} four 
times\footnote{Cases II${}_6$--II${}_9$ in Dynkin notation,
\cite{Dynkin}, no further inclusions arise in positive characteristics.} 
but each time in the embedding $\EE_7<\CC_{28}$. Formally, 
this only implies the maximality of $G(\EE_7,K)$ in $\Sp(56,K)$,
rather then the maximality of $\overline G(\EE_7,K)$ in
$\GSp(56,K)$. However, since $\det(h_{\varpi_7}(\eta))=\eta^{-28}$, 
for every algebraically closed field the determinant of 
$h_{\varpi_7}(\eta)$ can be arbitrary. Therefore, any connected
closed subgroup that properly contains $\overline G(\EE_7,K)$,
contains both $\Sp(56,K)$ and matrices of an arbitrary determinant,
and thus coincides with $\GSp(56,K)$. It only remains to observe that 
the group $\GSp(56,K)$ does not preserve the ideal $I$. Therefore,
$\overline G(\EE_7,K)$ is maximal among all such groups, and thus
coincides with $G_I(K)$.
\end{proof}

\begin{lemma}\label{lemma:g_fg_field}
Theorem~{\rm\ref{thm:g_fh}} holds for any algebraically closed field. 
\end{lemma}
\begin{proof}
Completely analogous to the proof of Lemma~\ref{lemma:g_i_field},
only that instead the reference to Theorem~\ref{thm:e7_in_fix}
one should invoke results of~\cite{Luzgarev_e7_invariants}, where
it is proven that $E(\EE_7,R)$ stabilizes the pair $(f,h)$. Therefore,
$G(\EE_7,K) = E(\EE_7,K)$ is contained in $G_{(f,h)}(K)$.
Moreover, it is easy to see that elements of the maximal torus
$T(\EE_8,K)$ act as similarities of the pair $(f,h)$, and thus
$\overline G(\EE_7,K)$ is contained in $\overline G_{(f,h)}(K)$.
\end{proof}


\section{Dimension of the Lie algebra}

In the present section we proceed with the proofs of Theorems~\ref{thm:g_i}
and~\ref{thm:g_fh}. Namely, here we prove that the affine group schemes 
$G_I$, $G_{(f,h)}$, $\overline G_{(f,h)}$ are smooth. This is one of the
key calculations in the present paper.

First, consider $G_I$. We should evaluate the dimension of the Lie algebra
of this scheme. It is well known how to calculate the Lie algebra that
stabilizes a system of forms, see, for instance~\cite{Jacobson_Jordan}).
Of course, before the advent of the theory of group schemes, in positive
characteristic it was not possible to derive any information concerning
the group stabilizing the same system of forms. Morally, our calculation
faithfully imitates the works by William Waterhouse, 
especially~\cite{Waterhouse_det}, where such similar calculations were 
performed in Lemmas 3.2, 5.3 and 6.3. Analogous calculations for the
cases of polyvector representation of $\GL_n$, and for the microweight
representation of $\EE_6$ were carried through 
in~\cite{Vavilov_Perelman, Vavilov_Luzgarev_norme6}.

Let, as above, $K$ be a field. The Lie algebra 
$\Lie((G_I)_K)$ of an affine group scheme $(G_I)_K$ is most naturally 
interpreted as the kernel of the homomorphism
$G_I(K[\delta])\longrightarrow G_I(K)$, sending $\delta$ to 0 
see~\cite{Borelrus, Humphreys_LAG, Vavilov_A3}. Let $G$ be a
subscheme of $\GL_n$. Then $\Lie(G_K)$ consists of all matrices
of the form $e+z\delta$, where $z\in M(n,K)$, satisfying the equations
defining $G(K)$. In the next lemma we specialize this statement in the
case where $G$ is the stabilizer of a system of polynomials. 

\begin{lemma}
Let $f_1,\ldots,f_s\in K[x_1,\ldots,x_t]$.
Then a matrix $e+z\delta$, where $z\in M(t,K)$, belongs to
$\Lie(\Fix_K(f_1,\ldots,f_s))$ if and only if
$$ 
\sum_{1\le i,j\le t}z_{ij}x_i \frac{\partial f_h}{\partial x_j}=0,
$$
for all $h=1,\ldots,s$.
\end{lemma}

The following result is proved in exactly the same way as 
Lemma~5.3 of~\cite{Waterhouse_det}, and as Theorem~4 
of~\cite{Vavilov_Luzgarev_norme6}. 
Clearly, the dimension that arises in this proof, is the 
dimension of the Lie algebra of type $\EE_7$ increased by 1. 
As also in \cite{Vavilov_Luzgarev_norme6}, from the proof it will 
be clear, which of the coefficients $z_{\lambda\mu}$ correspond
to roots, and which correspond to the Cartan subalgebra. 
The extra~1 is accounted by an additional toral summand,
since the Lie algebra we consider is in fact the Lie algebra
of the {\it extended\/} Chevalley group, whose dimension 
exceeds dimension of the ordinary Chevalley group by 1. 

\begin{theorem}\label{thm:Lie_dimension}
For any field $K$ the dimension of the Lie algebra $\Lie(G_I(K))$ 
does not exceed $134$.
\end{theorem}

\begin{proof}
Our equations take the form
\begin{align*}
\sum_{\lambda,\mu} z_{\lambda\mu}x_\lambda
\frac{\partial f_{\rho,\sigma}}{\partial x_\mu}&=0, 
\quad \rho,\sigma\in\Lambda,\;\rho\perp\sigma;
\\
\sum_{\lambda,\mu} z_{\lambda\mu}x_\lambda
\frac{\partial g_\alpha}{\partial x_\mu} &=0,
\quad\alpha\in\Phi.
\end{align*}
Recall, that the partial derivatives look as follows:
\begin{align*}
\frac{\partial f_{\rho,\sigma}}{\partial x_\mu} &=
\begin{cases}
\pm x_{\rho+\sigma-\mu}, &\text{if $\rho+\sigma-\mu\in\Lambda$,}\\
0,                       &\text{otherwise.}
\end{cases}\\
\frac{\partial g_\alpha}{\partial x_\mu} &=
\begin{cases}
\pm x_{\overline\mu}, &\text{if $\mu+\alpha\in\Lambda$ of
$\mu-\alpha\in\Lambda$,}\\
0,                    &\text{otherwise.}
\end{cases}
\end{align*}

\begin{itemize}
\item If $d(\lambda,\mu)=3$, then $z_{\lambda\mu}=0$. Indeed, in this case
  $\mu=\overline\lambda$. There exists a root $\alpha\in\Phi$ such that
  $\overline\lambda+\alpha\in\Lambda$ or $\overline\lambda-\alpha\in\Lambda$.
  Consider the equation corresponding to the polynomial $g_\alpha$. It 
  features the monomial
  $z_{\lambda\overline\lambda}x_\lambda
  \frac{\partial g_\alpha}{\partial x_{\overline\lambda}} =
  \pm z_{\lambda\overline\lambda}x_\lambda^2$. However, the monomial 
  $x_\lambda^2$ does not occur in any generator of the ideal
  $I$. It follows that the coefficient
  $\pm z_{\lambda\overline\lambda}$ must be 0.
\item If $d(\lambda,\mu)=2$, then $z_{\lambda\mu}=0$. Choose a root
  $\alpha\in\Phi$ such that $\mu+\alpha\in\Lambda$ or $\mu-\alpha\in\Lambda$.
  Consider the equation corresponding to the polynomial $g_\alpha$. It
  features the monomial
  $z_{\lambda\mu}x_\lambda\frac{\partial g_\alpha}{\partial x_\mu} = 
  \pm z_{\lambda\mu}x_\lambda x_{\overline\mu}$, where
  $d(\lambda,\overline\mu)=1$. Thus, the monomial $x_\lambda x_{\overline\mu}$ 
  does not occur in any generator of the ideal $I$. The non-zero summand
  $\pm z_{\lambda\mu}x_\lambda x_{\overline\mu}$ could only possible cancel with 
  the non-zero summand of the form
  $z_{\overline\mu\overline\lambda}x_{\overline\mu}\frac{\partial g_\alpha}{\partial
    x_{\overline\lambda}}$. However, varying $\alpha$ one can guarantee that both
  $\overline\lambda+\alpha\notin\Lambda$ and
  $\overline\lambda-\alpha\notin\Lambda$. For instance, by the transitivity
  of the Weyl groups on pairs of weights at distance 2, one can assume that
  $\lambda=\varpi_7$ and $\mu=\overline\varpi_7+\alpha_7$, in which case one
  can take $\alpha=\alpha_6$. For this choice of $\alpha$ the summand 
  $\pm z_{\lambda\mu} x_\lambda x_{\overline\mu}$ remains the summand containing
  $x_\lambda x_{\overline\mu}$, and thus $z_{\lambda\mu}=0$.
\item If $d(\lambda,\mu) = d(\nu,\rho) = 1$ and $\lambda-\mu=\nu-\rho$, then
  $z_{\lambda\mu}=\pm z_{\nu\rho}$.
  First, assume that $(\nu,\rho)\neq (\overline\mu,\overline\lambda)$.
  In this case, $\lambda\perp\rho$, $\mu\perp\nu$ and $\lambda+\rho=\mu+\nu$.
  Consider the equation corresponding to the polynomial $f_{\lambda,\rho} =
  f_{\mu,\nu}$. It features the monomials
  $z_{\lambda\mu}x_\lambda\frac{\partial f_{\lambda,\rho}}{\partial x_\mu} =
  \pm z_{\lambda\mu}x_\lambda x_\nu$ and
  $z_{\nu\rho}x_\nu\frac{\partial f_{\lambda,\rho}}{\partial x_\rho} =
  \pm z_{\nu\rho}x_\nu x_\lambda$.
  However, $d(\lambda,\nu)=1$, so that the monomial 
  $x_\lambda x_\nu$ does not occur in any generator of the ideal 
  $I$. This means that these two monomials must sum to $0$, so that
  $z_{\lambda\mu} = \pm z_{\nu\rho}$.
\item If $d(\lambda,\mu) = d(\nu,\rho) = 1$ and $\lambda-\mu = \nu-\rho$, then
  $z_{\rho\rho} = \pm z_{\lambda\lambda} \pm z_{\mu\mu} \pm
  z_{\nu\nu}$. Indeed, the equation corresponding to the same polynomial 
  $f_{\lambda,\rho} = f_{\mu,\nu}$, as in the preceding item, features
  monomials 
  $z_{\lambda\lambda}x_\lambda\frac{\partial f_{\lambda,\rho}}{\partial x_\lambda}
  = \pm z_{\lambda\lambda}x_\lambda x_\rho$,
  $z_{\mu\mu}x_\mu\frac{\partial f_{\lambda,\rho}}{\partial x_\mu}
  = \pm z_{\mu\mu}x_\mu x_\nu$,
  $z_{\nu\nu}x_\nu\frac{\partial f_{\lambda,\rho}}{\partial x_\nu}
  = \pm z_{\nu\nu}x_\nu x_\mu$, and
  $z_{\rho\rho}x_\rho\frac{\partial f_{\lambda,\rho}}{\partial x_\rho}
  = \pm z_{\rho\rho}x_\rho x_\lambda$. Observe, that the monomial 
  $x_\lambda x_\rho$ occurs in exactly one of the generators of $I$, viz. in 
  $f_{\lambda,\rho}$), and $x_\mu x_\nu$ occurs in the same polynomial, with the
  same coefficient, up to sign. Equating the coefficients of the above
  monomials, we see
  $z_{\lambda\lambda}\pm z_{\rho\rho} = \pm z_{\mu\mu} \pm z_{\nu\nu}$.
\end{itemize}

Let us summarize what we have just established. The first two items 
show that the matrix entries $z_{\lambda\mu}$ with $d(\lambda,\mu)\geq 2$ 
do not contribute to the dimension of the Lie algebra, whereas the 
entries $z_{\lambda\mu}$ with $d(\lambda,\mu) = 1$ give the contribution
equal to the number of roots of $\Phi$, namely, $126$. Finally, the last 
item allows us to express all entries $z_{\lambda\lambda}$ as linear
combinations of the entries $z_{\mu\mu}$, for $\mu=\mu_1,\dots,\mu_t$, 
such that each fundamental root of $\Phi$ occurs among the pair-wise
differences of the weights $\mu_i$. It is easy to see that the
smallest
number of such weights is 8, and that one can use the weights 
$1,2,3,4,5,6,7,8$ as such. Figure~3 shows their location in the weight diagram. 

\begin{figure}[h]
\begin{center}
\includegraphics{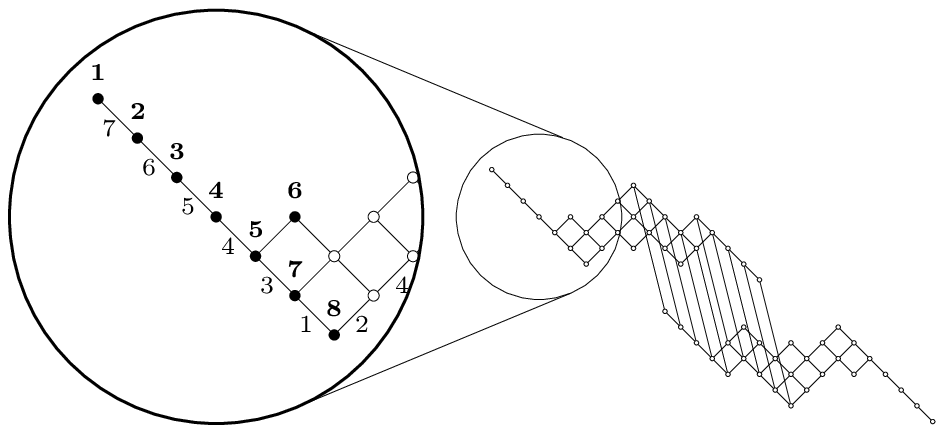}
\caption{}
\end{center}
\end{figure}

Thus, the dimension of the Lie algebra does not exceed $126+8=134$.
\end{proof}

Next, we pass to the schemes $G_{(f,h)}$ and $\overline G_{(f,h)}$.
As above, we can identify the Lie algebras
$\Lie(G_{(f,h)}(K))$ and $\Lie(\overline{G}_{(f,h)}(K))$ with the
kernels of homomorphisms obtained by specializing $\delta$ in the
ring of dual numbers $K[\delta]$ to 0. Thus, 
$\Lie(G_{(f,h)}(K))$ consists of the matrices $g=e+x\delta$,
where $x\in M(n,K)$, satisfying the following conditions:
$f(gu,gv,gw,gz) = f(u,v,w,z)$ and $h(gu,gv) = h(u,v)$,
for all $u,v,w,z\in V$.
Similarly, $\Lie(\overline{G}_{(f,h)}(K))$ consists of all matrices 
$g=e+x\delta$, where $x\in M(n,K)$, satisfying the conditions
$f(gu,gv,gw,gz) = \eps(g)f(u,v,w,z)$ and $h(gu,gv) = \eps^2(g)h(u,v)$
for all $u,v,w,z\in V$.

\begin{theorem}\label{thm:Lie_dimension_fh}
For any field $K$ the dimension of the Lie algebra $\Lie(G_{(f,h)}(K))$
does not exceed $133$, while the dimension of the Lie algebra 
$\Lie(\overline{G}_{(f,h)}(K))$ does not exceed $134$.
\end{theorem}

\begin{proof}
First, observe that the conditions on elements of the Lie algebra
$\Lie(G_{(f,h)}(K))$ are obtained from the corresponding conditions 
for elements of $\Lie(\overline{G}_{(f,h)}(K))$ by substituting
$\eps(g) = \eps'(g) = 1$ and $c_2(g) = c_3(g) = c_4(g) = 0$.
Let $g$ be a matrix satisfying the above conditions for all
$u,v,w,z\in V$. Plugging in $g = e+x\delta$ and using that the form
$f$ is four-linear, whereas the form $h$ is bilinear, we get
\begin{align*}
f(xu,v,w,z) &+ f(u,xv,w,z) + f(u,v,xw,z) + f(u,v,w,xz) 
\\
&= (\eps(g)-1)f(u,v,w,z) + c_2(g)h(u,v)h(w,z) 
\\
&\quad + c_3(g)h(u,w)h(v,z) + c_4(g)h(u,z)h(v,w)
\end{align*}
and
\begin{align*}
h(xu,v) + h(u,xv) &= (\eps'(g)-1)h(u,v).
\end{align*}
Now we show that the entries of the matrix $x$ are subject to exactly the
same linear dependences, as in the proof of Theorem~\ref{thm:Lie_dimension}.
\begin{itemize}
\item 
If $d(\lambda,\mu)=3$, then $x_{\lambda\mu}=0$.
  Indeed, in this case $\mu = \overline{\lambda}$.
  Let $(\lambda,\rho,\sigma,\tau)$ be a non-degenerate quadruple of
  weights containing $\lambda$. Set
  $u = e_\rho$, $v = e_\sigma$, $w = e_\tau$, $z = e_{\mu}$.
  Then $(\rho,\mu) = (\rho,\delta-\lambda) = (\rho,\delta) = 1/2$.
  Similarly, $(\sigma,\mu) = (\tau,\mu) = 1/2$. This means that $\mu$
  is not orthogonal and not opposite to any of the weights 
  $\rho,\sigma,\tau$.
  Thus $f(xu,v,w,z) = f(u,xv,w,z) = f(u,v,xw,z) = f(u,v,w,z) = 0$.
  It follows that $f(u,v,w,xz) = \pm x_{\lambda\mu} = 0$.
\item 
If $d(\lambda,\mu)=2$, then $x_{\lambda\mu}=0$.
  Let $\{\lambda\}^\perp\subseteq\Lambda$ be the set of weights 
  orthogonal to $\lambda$. Observe that $\mu\in\{\lambda\}^\perp$.
  Moreover,  $|\{\lambda\}^\perp| = 27$, and these weights are
  the weights of the $27$-dimensional representation of the Chevalley 
  group of type $\EE_6$. Take in $\{\lambda\}^\perp$ three weight 
  $\rho,\sigma,\tau$ forming a triad (i.e.\ pair-wise orthogonal) 
  in such a way that $\mu\notin\{\rho,\sigma,\tau\}$. Then
  $(\lambda,\rho,\sigma,\tau)$ is a non-degenerate quadruple of
  weights. Set $u=e_\rho$, $v = e_\sigma$, $w = e_\tau$, $z = e_\mu$.
  Observe that $\mu$ cannot be orthogonal to more than one of the
  weights $\rho,\sigma,\tau$. Indeed, let $\mu$ be orthogonal to two of
  them, say $\mu\perp\rho$ and $\mu\perp\sigma$. Since
  $\mu\in\{\lambda\}^\perp$, and in $\{\lambda\}^\perp$ there
  is a unique weight that is orthogonal to both $\rho$ and $\sigma$,
  namely, $\tau$. It follows that $\mu=\tau$, a contradiction.
  Thus,  $f(xu,v,w,z) = f(u,xv,w,z) = f(u,v,xw,z) = f(u,v,w,z) = 0$.
  It follows that $f(u,v,w,xz) = \pm x_{\lambda\mu} = 0$.
\item 
If $d(\lambda,\mu)=1$ and $\lambda-\mu=\nu-\rho$, then
  $x_{\lambda\mu} = \pm x_{\nu\rho}$. By the transitivity of the action 
  of the Weyl group on the pairs of weights at distance 1, we can
  move the pair $(\lambda,\mu)$ to the pair $(1,2)$. Then 
  $\nu-\rho=\alpha_7$. In the weight diagram there are exactly 
  $12$ edges marked $7$, which gives us exactly $12$ possibilities 
  for the pair $(\nu,\rho)$. This leaves us with the following three
  cases to examine:
  \begin{enumerate}
    \item $(\nu,\rho) = (\lambda,\mu)$;
    \item $d(\lambda,\rho)=2$;
    \item $(\nu,\rho) = (\overline\mu,\overline\lambda)$.
  \end{enumerate}
  The first of these cases is trivial. Next, we observe that is suffices to 
  prove the equality 
  $x_{\lambda\mu} = \pm x_{\nu\rho}$ for the second case, and then to
  use the chain of equalities
  $x_{\lambda\mu} = \pm x_{\nu\rho} = \pm x_{\overline\mu \overline\lambda}$
  to derive the third case. Thus, we are left with the analysis of the
  situation, where $d(\lambda,\rho) = 2$. There exists a non-degenerate 
  quadruple of weights $(\lambda,\rho,\sigma,\tau)$. Furthermore, we can
  stipulate that $d(\mu,\sigma)=2$. Indeed, in the case $\lambda=\lambda_1$, 
  $\mu=\lambda_2$,
  $\rho=\lambda_{\overline{16}}$, the choice $\sigma=\lambda_{\overline{15}}$,
  $\tau=\lambda_{\overline{14}}$ would do. Now, set
  $u=e_\mu$, $v=e_\rho$, $w=e_\sigma$, $z=e_\tau$.
  Then $d(\mu,\rho)=1$, and thus
  $f(u,v,xw,z) = f(u,v,w,xz) = f(u,v,w,z) = 0$.
  It follows that $f(xu,v,w,z) + f(u,xv,w,z) = 0$.
 But $f(xu,v,w,z) = \pm x_{\lambda\mu}$ and $f(u,xv,w,z) = \pm x_{\nu\rho}$.
\item 
If $d(\lambda,\mu)=1$ and $\lambda-\mu=\nu-\rho$, then
  $x_{\lambda\lambda} - x_{\mu\mu} = x_{\nu\nu} - x_{\rho\rho}$.
  As in the proof of the preceding item, we can limit ourselves
  with the analysis of the case, where $d(\lambda,\rho)=2$. Again we can
  find a non-degenerate quadruple of weights $(\lambda,\rho,\sigma,\tau)$ 
  such that
  $d(\mu,\sigma)=2$. Setting $u=e_\lambda$, $v=e_\rho$, $w=e_\sigma$,
  $z=e_\tau$, we get $(x_{\lambda\lambda} + x_{\rho\rho} +
  x_{\sigma\sigma} + x_{\tau\tau} - \eps(g) +
  1)f(e_\lambda,e_\rho,e_\sigma,e_\tau) = 0$. It follows that 
  $x_{\lambda\lambda}+x_{\rho\rho}+x_{\sigma\sigma}+x_{\tau\tau}=\eps(g)-1$.
  On the other hand, $(\mu,\nu,\sigma,\tau)$ is another non-degenerate
  quadruple. Setting $u=e_\mu$, $v=e_\nu$, $w=e_\sigma$,
  $z=e_\tau$, we get 
  $(x_{\mu\mu} + x_{\nu\nu} + x_{\sigma\sigma} + x_{\tau\tau} -\eps(g) +
  1)f(e_\mu,e_\nu,e_\sigma,e_\tau) = 0$. It follows that 
  $x_{\mu\mu} + x_{\nu\nu} + x_{\sigma\sigma} + x_{\tau\tau} = \eps(g) - 1$.
  Comparing these expressions, we can conclude that
  $x_{\lambda\lambda}-x_{\mu\mu} = x_{\nu\nu}-x_{\rho\rho}$.
\end{itemize}

Thus, as in the proof of Theorem~\ref{thm:Lie_dimension}, it turns
out that the dimension of the Lie algebra $\Lie(\overline{G}_{(f,h)}(K))$
does not exceed $126+8=134$. The same arguments are also applicable
for the case of $\Lie(G_{(f,h)}(K))$. It suffices to set
$\eps(g)=\eps'(g)=1$ and $c_2(g)=c_3(g)=c_4(g)=0$. Again, we can conclude 
that the dimension of $\Lie(G_{(f,h)}(K))$ does not exceed $134$: the
entries $x_{\lambda\mu}$ do not contribute to the dimension when
$d(\lambda,\mu)\geq 2$, they make a contribution $126$, when $d(\lambda,\mu)
= 1$, and, finally, they make a contribution $\le 8$, for $d(\lambda,\mu) = 0$.

To conclude the proof of the theorem, we have to find a non-trivial relation 
among these last entries. From the final paragraph of the proof
of Theorem~\ref{thm:Lie_dimension}) we know that the entries 
$z_{\lambda\lambda}$ are linear combinations of 8 of them, namely
$z_{\mu\mu}$, for $\mu = \mu_1,\dots,\mu_8$. Here, as 
$\mu_1,\dots,\mu_8$ one can takes the weights
$1,2,3,4,5,6,7,8$, respectively. Now, set 
$u = e_{\mu_1}$, $v = e_{\overline{\mu_1}}$. Plugging these entries into the
equation 
$h(xu,v) + h(u,xv) = 0$, we get that  $x_{\mu_1\mu_1} + x_{\overline{\mu_1}\overline{\mu_1}}
= 0$.
One the other hand, $\mu_1 - \overline{\mu_1} = \delta = (2\alpha_1 +
3\alpha_2 + 4\alpha_3 + 6\alpha_4 + 5\alpha_5+ 4\alpha_6 + 3\alpha_7)$,
and thus the relations
$x_{\lambda\lambda} - x_{\mu\mu} = x_{\nu\nu} - x_{\rho\rho}$
for $\lambda-\mu = \nu-\rho\in\Phi$
imply that
\begin{align*}
x_{\mu_1\mu_1} - x_{\overline{\mu_1}\overline{\mu_1}} &=
2(x_{\mu_7\mu_7} - x_{\mu_8\mu_8})
+ 3(x_{\mu_5\mu_5} - x_{\mu_6\mu_6}) 
\\
&\quad + 4(x_{\mu_5\mu_5} - x_{\mu_7\mu_7})
 + 6(x_{\mu_4\mu_4} - x_{\mu_5\mu_5}) 
\\
&\quad + 5(x_{\mu_3\mu_3} - x_{\mu_4\mu_4})
 + 4(x_{\mu_2\mu_2} - x_{\mu_3\mu_3}) 
\\
&\quad + 3(x_{\mu_1\mu_1} - x_{\mu_2\mu_2}) 
\\
&= 3x_{\mu_1\mu_1} + x_{\mu_2\mu_2}
  + x_{\mu_3\mu_3} + x_{\mu_4\mu_4} 
\\ 
&\quad + x_{\mu_5\mu_5} - 3x_{\mu_6\mu_6}
  - 2x_{\mu_7\mu_7} - 2x_{\mu_8\mu_8}.
\end{align*}
Comparing this with the equality $x_{\mu_1\mu_1} + x_{\overline{\mu_1}\overline{\mu_1}} =
0$, we get that
$x_{\mu_1\mu_1} + x_{\mu_2\mu_2} + x_{\mu_3\mu_3} + x_{\mu_4\mu_4} +
x_{\mu_5\mu_5} - 3x_{\mu_6\mu_6} - 2x_{\mu_7\mu_7} - 2x_{\mu_8\mu_8} =
0$.
This is precisely the desired non-trivial linear relation among the 
elements $x_{\mu_i\mu_i}$, which, over a field of any characteristic,
shows that the dimension of our Lie algebra is $1$ smaller than the 
above bound. Thus, $\dim\Lie(G_{(f,h)}(K)) \leq 133$, as claimed.
\end{proof}


\section{Proofs of Theorems~\ref{thm:g_i} and~\ref{thm:g_fh}}

Now we are all set to finish the proofs of our main results.

\begin{proof}[Proof of Theorem~\ref{thm:g_i}]
Consider the rational representation of algebraic groups
$$
\pi\colon\overline{G}(\EE_7,{-})\to\GL_{56}
$$
with the highest
weight $\varpi_7$. This representation is faithful, and by
Theorem~\ref{thm:e7_in_fix} its image is contained in~$G_I$. 
We wish to apply to this morphism $\pi$ Lemma~\ref{lem:Waterhouse}.

Indeed, for an algebraically closed field $K$ and for $K[\delta]$
the representation $\pi$ is a monomorphism. This means that the
condition 2 of Lemma~\ref{lem:Waterhouse} holds. Clearly,
$\dim(\overline{G}(\EE_7,K)) = 134$, and Theorem~\ref{thm:Lie_dimension} 
implies that also $\dim_K(\Lie(G_K))\leq 134$, so that the condition 1
of Lemma~\ref{lem:Waterhouse} follows from the fact that by 
Lemma~\ref{lem:irreducible} already the normalizer of $\overline{G}(\EE_7,K)$ 
in $\GL(56,K)$ is contained in --- and in fact coincides with --- $G_I(K)$. 
This means that we can apply Lemma~\ref{lem:Waterhouse} to conclude
that $\pi$ establishes an isomorphism of $\overline{G}(\EE_7,{-})$ 
and $G_I$, as affine group schemes over $\Int$.
\end{proof}

\begin{proof}[Proof of Theorem~\ref{thm:g_fh}]
Here again we can use Lemma~\ref{lem:Waterhouse}. The situation is
completely analogous to the proof of Theorem~\ref{thm:g_i}, only that
instead of reference to Theorem~\ref{thm:e7_in_fix}, one should 
invoke the main theorem of~\cite{Luzgarev_e7_invariants}, and
instead of Theorem~\ref{thm:Lie_dimension} one should use 
Theorem~\ref{thm:Lie_dimension_fh}.
\end{proof}

\begin{lemma}\label{lemma:Petrov}
In the definition of the group $\overline{G}_{(f,h)}(R)$ always 
$\eps'(g) = (\eps(g))^2$ and
$c_2(g) = c_3(g) = c_4(g) = 0$. In other words,
\begin{align*}
\overline{G}_{(f,h)}(R) &= \{g\in\GL(56,R)\mid\text{ there exists
  $\eps\in R^*$ such that }
\\
&\qquad f(gu,gv,gw,gz)=\eps f(u,v,w,z)
\\
&\qquad \text{and }h(gu,gv)=\eps^2h(u,v)\text{ for all }
u,v,w,z\in V\}.
\end{align*}
\end{lemma}

\begin{proof}
We have already shown that the group $\overline{G}_{(f,h)}(R)$ 
coincides with the extended Chevalley group $\overline{G}(\EE_7,R)$.
Clearly, both $\eps$ and $\eps'$ are homomorphisms from
$\overline{G}(\EE_7,{-})$ to $\mathbb G_m$, trivial on the commutator subgroup.
Thus, their appropriate powers should coincide. Calculating their values on
the semi-simple element $h_{\varpi_7}(\eta)$, we see that $\eps$ 
takes the value $\eta^2$, whereas $\eps'$ takes the value $\eta$. 
Thus, $\eps = (\eps')^2$.

On the other hand, it is easy to see that $c_i$ enjoy the 1-cocycle 
identity
$c_i(gh) = \eps(g)c_i(h) + c_i(g)\eps(h)$ 
and vanish on both the commutator subgroup and the semi-simple elements 
of the form $h_{\varpi_7}(\eta)$. As an algebraic group, the
extended Chevalley group is generated by these two subgroups,
so that $c_i$ are identically 0.
\end{proof}


\section{Proof of Theorem~\ref{thm:normaliser}}

\begin{proof}[Proof of Theorem~\ref{thm:normaliser}] 
Clearly, $\overline{G}(\EE_7,R)\leq N(G(\EE_7,R))$.
It is well known, see, for instance~\cite{Hazrat_Vavilov}) and
references there, 
that for any irreducible root system $\Phi$ of rank strictly larger 
than 1, and for any commutative ring $R$ the elementary group 
$E(\Phi,R)$ is normal in the extended Chevalley group
$\overline{G}(\Phi,R)$. Therefore,
$\overline{G}(\EE_7,R)\leq N(E(\EE_7,R))$. On the other hand,
both normalizers $N(E(\EE_7,R))$ and $N(G(\EE_7,R))$ are obviously
contained in the transporter $\Tran(E(\EE_7,R),G(\EE_7,R))$. Thus,
to finish the proof of the theorem, it suffices to verify that
$\Tran(E(\EE_7,R),G(\EE_7,R))$ is contained in~$\overline{G}(\EE_7,R)$.

Let $g\in\GL(56,R)$ belong to $\Tran(E(\EE_7,R),G(\EE_7,R))$.
WE pick any root $\alpha\in\Phi$ and any $\xi\in R$.
Then $a = gx_{\alpha}(\xi)g^{-1}$ lies in $G(\EE_7,R)$, and thus
$f(au,av,aw,az)=f(u,v,w,z)$ and $h(au,av)=h(u,v)$
for all  $u,v,w,z\in V$. Therefore, substituting $(gu,gv,gw,gz)$ for
$(u,v,w,z)$, we get
\begin{align*}
&f(gx_\alpha(\xi)u, gx_\alpha(\xi)v, gx_\alpha(\xi)w, gx_\alpha(\xi)z)
\\
&\quad = f(gu,gv,gw,gz)\text{ for all } u,v,w,z\in V.
\end{align*}
Consider the form $F\colon V\times V\times V\times V\to R$,
defined by
$$
F(u,v,w,z) = f(gu,gv,gw,gz).
$$
By our assumption, one has
$$
F(x_\alpha(\xi)u,x_\alpha(\xi)v,x_\alpha(\xi)w,x_\alpha(\xi)z)
= F(u,v,w,z)
$$
for all $u,v,w,z\in V$ and for all $\alpha\in\Phi$, $\xi\in R$.
Root unipotents $x_\alpha(\xi)$ generate the elementary group
$E(\EE_7,R)$. It follows that the form $F$ is invariant under the action 
of this group. Obviously, the form $F$ is four-linear.
Thus, we can apply to this form the main result 
of~\cite[Theorem~2]{Luzgarev_e7_invariants}). It says that 
in this case the form $F$ has the shape
\begin{align*}
F(u,v,w,z) &= \eps f(u,v,w,z) + c_2 h(u,v)h(w,z) 
\\
&+ c_3 h(u,w)h(v,z) + c_4 h(u,z)h(v,w),
\end{align*}
for some $\eps,c_2,c_3,c_4\in R$. Plugging in $g^{-1}$
instead of $g$, we can conclude that $\eps\in R^*$.

A similar calculation for $h$ shows that $h(gu,gv) = \eps' h(u,v)$
for some $\eps'\in R$. Again, plugging in $g^{-1}$ instead of
$g$, we can conclude that $\eps'\in R^*$.

This shows that $g$ belongs to the group $\overline{G}_{(f,h)}(R)$,
which by Theorem~\ref{thm:g_fh} coincides with $\overline{G}(\EE_7,R)$.
\end{proof}

The authors are grateful to Ernest Borisovich Vinberg, who pointed
out a serious error in a preliminary version of this paper.
Also, the authors are grateful to the referee, for the statement
and proof of Lemma~\ref{lemma:Petrov}.

\end{document}